\documentclass{amsart}

\usepackage{amssymb, tikz}
\usetikzlibrary{shapes}
\usetikzlibrary{calc,arrows.meta,shadows,arrows,decorations.pathmorphing,backgrounds,positioning,fit,petri, hobby}
\usepackage{caption}
\usepackage{enumitem}
\usepackage{hyperref}
\hypersetup{
    unicode,
    colorlinks,
    linkcolor={red!60!black},
    citecolor={green!60!black},
    urlcolor={blue!60!black}
}
\usepackage{subcaption}

\usepackage{wrapfig}

\usepackage{listings}
\usepackage{setspace}
\definecolor{Code}{rgb}{0,0,0}
\definecolor{Decorators}{rgb}{0.5,0.5,0.5}
\definecolor{Numbers}{rgb}{0.5,0,0}
\definecolor{MatchingBrackets}{rgb}{0.25,0.5,0.5}
\definecolor{Keywords}{rgb}{0,0,1}
\definecolor{self}{rgb}{0,0,0}
\definecolor{Strings}{rgb}{0,0.63,0}
\definecolor{Comments}{rgb}{0,0.63,1}
\definecolor{Backquotes}{rgb}{0,0,0}
\definecolor{Classname}{rgb}{0,0,0}
\definecolor{FunctionName}{rgb}{0,0,0}
\definecolor{Operators}{rgb}{0,0,0}
\definecolor{Background}{rgb}{0.98,0.98,0.98}

\newtheorem{theorem}{Theorem}
\numberwithin{theorem}{subsection}
\newtheorem{lemma}[theorem]{Lemma}
\newtheorem{cor}[theorem]{Corollary}
\newtheorem{prop}[theorem]{Proposition}

\newtheorem{prob}{Problem}

\newtheorem{exam}[theorem]{Example}

\newcommand{\Z}{\mathbb{Z}}
\newcommand{\N}{\mathbb{N}}

\newcommand{\script}{\mathcal}
\newcommand{\parentheses}[1]{{( {#1})}}

\newcommand{\p}{\parentheses}

\newcommand{\closure}[1]{\overline{#1}}

\newcommand{\Set}[1]{{\lbrace {#1} \rbrace}}

\newcommand{\cardinality}[1]{{\lvert {#1}\rvert}}

\def\set#1:#2{\Set{{#1} \colon {#2}}}

\newcommand{\Fr}[1]{\digamma\!#1}

\newcommand{\hideme}[1]{}

\title{$n$-arc and $n$-circle connected graph-like spaces}
\author{Paul Gartside}
\address{Department
    of Mathematics, University of Pittsburgh, Pittsburgh, PA~15260, USA}
\email{gartside@math.pitt.edu} 

\author{Max Pitz}
\address{Department of Mathematics, University of Hamburg, Bundesstra\ss e 55, 20146 Hamburg, Germany}
\email{max.pitz@uni-hamburg.de}



\begin{document}

\begin{abstract} A  space  $X$ is \emph{$n$-arc connected} (respectively, \emph{$n$-circle connected})   if for any choice of at most $n$ points there is an arc (respectively, a circle) in $X$ containing the specified points.
We study $n$-arc connectedness and $n$-circle connectedness in compactifications of locally finite graphs and the
slightly more general class of graph-like continua, uncovering a striking difference in their behaviour regarding $n$-arc and -circle connectedness.
\end{abstract}

\maketitle

\section{Introduction}

A topological space  $X$ is \emph{$n$-arc connected}, abbreviated $n$-ac,  if for any choice of at most $n$ points there is an arc (a homeomorph of the closed unit interval) in $X$ containing the specified points. Similarly, $X$ is \emph{$n$-circle connected} (abbreviated, $n$-cc) if for any choice of at most $n$ points there is a simple closed curve (homeomorph of the unit circle) in $X$ containing the specified points. Note that a space is arc connected if and only if it is $2$-ac.  
A space which is $n$-ac (respectively, $n$-cc) for all $n$ is called $\omega$-ac (respectively, $\omega$-cc).

Every graph is a topological space when considered as a $1$-complex, and recently the authors together with A.\ Mamatelashvili, developing results from \cite{acpaper}, have given a complete combinatorial characterization of which graphs (without any restriction on the number of vertices, or edges, or the degree of any vertex) are $n$-ac or $n$-cc for any $n \in \N$, see \cite{GMP}. 
In particular, a non-degenerate graph $G$ is $7$-ac if and only if it is $\omega$-ac if and only if $G$ is homeomorphic to one of nine distinct graphs \cite[Theorem~3.5.1]{GMP}. For $n \leq 6$ there are infinitely many $n$-ac graphs (even finite), but effective characterizations are now known. For example \cite[Theorem~3.4.1]{GMP}: a graph $G$ is $6$-ac if and only if either $G$ is one of the nine $7$-ac graphs mentioned above, or, after suppressing all degree-2-vertices, the combinatorial graph $G$ is $3$-regular, $3$-connected, and removing any $6$ edges does not disconnect $G$ into $4$ or more components. When considering $n$-cc graphs, the situation is even simpler: the only $3$-cc graphs are the finite cycles, while $2$-cc graphs are those that contain no cut vertices.

Finite graphs are extremely simple continua (a \emph{continuum} is a compact, metric and connected space), and for arbitrary continua the problem of characterizing which are $n$-ac or $n$-cc is difficult. Indeed, using ideas from descriptive set theory, it is shown in \cite{sacpaper} that there is no characterization of $n$-ac rational continua simpler than the definition of $n$-ac (here $n$ is in $\N \cup \{\omega\}$, and a continuum is \emph{rational} if it has a base of open sets whose boundaries are countable). 

It is natural to investigate where the transition between the results for graphs -- `$7$-ac implies $\omega$-ac' and effective characterizations for $n \le 6$ -- and the provable complexity for rational continua occurs. 
In \cite{FD}, for each $n$, a regular continuum is constructed which is $n$-ac but not $(n+1)$-ac (a continuum is \emph{regular} if it has a base of open sets whose boundaries are finite). So, in this context, regular continua are too complex.

In the present paper it is shown that the transition takes place precisely between the Freudenthal compactification of locally finite graphs and graph-like continua. Graph-like continua were introduced as a natural abstraction of the Freudenthal compactification of locally finite graphs. Up until now all results about the Freudenthal compactification of locally finite graphs have extended naturally to graph-like continua.
Thus, it was entirely unanticipated that the $n$-ac property behaves so differently between the Freudenthal compactification of locally finite graphs and graph-like continua. 

\subsection{Freudenthal compactification of locally finite graphs}

Let $G$ be a locally finite, countable, connected graph. Its \emph{Freudenthal compactification}, denoted $\Fr{G}$, is the maximal compactification of $G$ with \emph{zero-dimensional} remainder, $\Fr{G} \setminus G$. (See the discussion immediately preceding Theorem~\ref{lem_wlogpointsonedges2Freudenthal} below for an alternative, constructive description of the Freudenthal compactification of a locally finite graph.) A space is \emph{zero-dimensional} if it has a basis of open sets whose boundaries are empty, i.e. a basis of set which are simultaneously closed and open (clopen).

In the last two decades, Diestel and his students have shown that many combinatorial theorems about paths and cycles in finite graphs extend verbatim to the Freudenthal compactification of infinite, locally finite graphs if one exchanges finite paths and cycles for topological arcs and simple closed curves respectively, see \cite[Chapter 8]{Diestel} and \cite{DSurv}. 

Given this evidence, it might not come as a surprise that the property of $n$-arc connectedness also lifts nicely to the Freudenthal compactification. Indeed, as our first main result of this paper, we show in Theorem~\ref{lem_wlogpointsonedges2Freudenthal} that for a locally finite, connected graph $G$ and some $n \in \mathbb{N}$, its Freudenthal compactification $FG$ is $n$-ac [$n$-cc] if and only if $G$ itself is $n$-ac [$n$-cc], allowing us to lift all our characterizations from \cite{GMP}. However, we also give examples that this is not generally true for all compactifications with zero-dimensional remainder, and it remains an open problem, for example, to characterize for which locally finite graphs the one-point compactification is $n$-ac. What remains true, though, is the fact that there are only six different $7$-ac graph compactifications, all of which all are again even $\omega$-ac. So there is no jump in complexity happening at this point yet. These results are in Section~\ref{sec_2}.

\subsection{Graph-like continua}
A \emph{graph-like continuum} is a continuum $X$ which contains a closed zero-dimensional  subset  $V$, such that for some  discrete index set $E$ we have  that $X \setminus V$ is homeomorphic to $E \times (0,1)$. The sets $V$ and $E$ are the \emph{vertices} and \emph{edges} of $X$ respectively. Clearly a compactification of a connected, 
locally finite graph is graph-like if and only if the remainder is zero-dimensional. Thus the Freudenthal compactification is graph-like. 

In fact, graph-like spaces were introduced by Thomassen and Vella as a natural abstraction of the Freudenthal compactification of a graph, in order to eliminate the necessity for distinct treatments of vertices and ends in arguments about $\Fr{G}$. Papers in which graph-like spaces have played a key role include \cite{thomassenvella} where several Menger-like results are given, and  \cite{graphlikeplanar} where algebraic criteria for the planarity of graph-like spaces are presented. In \cite{infinitematroids}, aspects of the matroid theory for graphs have been generalized to infinite matroids on graph-like spaces.

We now know from \cite[Theorem A]{EGP} that graph-like continua had earlier been studied by topologists under the name \emph{completely regular continua} (continua in which every non-degenerate subcontinuum has non-empty interior), and are much closer both to finite graphs and the Freudenthal compactification of graphs than their definition `by analogy' might suggest. 
Indeed a continuum is graph-like if and only if it the inverse image of finite graphs under edge-contraction bonding maps 
(see Section~\ref{mach_gl} for details), if and only if it is a (standard) subcontinuum of a Freudenthal compactification of a graph.

%

Even though the graph-like continua are in complexity just a small step above compactifications of locally finite graphs, it turns out that this is already enough to give rise to completely new and surprising examples of $n$-ac and $n$-cc graph-like continua for \emph{all} $n \ge 2$ and $\omega$.
For $n$-circle connectedness, our main result is as follows: while there is topologically a unique $3$-cc graph compactification, namely the circle (which is even $\omega$-cc), we show in Theorem~\ref{thm_manydifferentomegacc} that there are in fact continuum, $2^{\aleph_0}$,  many pairwise non-homeomorphic $\omega$-cc graph-like continua. 
For $n$-arc connectedness, our main result is: while there are only six different $7$-ac graph compactifications (which all are even $\omega$-ac), we show in Theorem~\ref{thm_manydifferentother} that for every $n \ge 2$ there are continuum many $n$-ac [$n$-cc] graph-like continua which are not $(n+1)$-ac [$(n+1)$-cc]. 

These examples are presented in Section~\ref{sec_4}.
In Section~\ref{sec_3} we develop the necessary machinery to construct graph-like continua, and to check whether they are $n$-ac or $n$-cc. In addition -- and as an exception to the rule -- the $2$-cc graph-like continua are characterized, just like graphs, as being those without cut points, and as having inverse limit representations by finite $2$-cc graphs.  

\section{Locally Finite Graphs, and their Freudenthal Compactification}
\label{sec_2}

The fundamental result of this section is Theorem~\ref{lem_wlogpointsonedges2Freudenthal} stating that the Freudenthal compactification $\Fr{G}$ of a  locally finite graph $G$ is $n$-ac precisely when $G$ is $n$-ac. Since the problem of determining when a graph is $n$-ac, or $n$-cc, is completely solved, so is the problem for Freudenthal compactifications of locally finite graphs. 

Parts of these results can be extended to arbitrary graph-like compactifications of locally finite graphs. But examples demonstrate that Theorem~\ref{lem_wlogpointsonedges2Freudenthal}  does not extend in full generality to graph-like compactifications of locally finite graphs.

\subsection{Restricting to points on edges.}

We begin with the following extension of \cite[Lemma 2.3.5]{GMP} 
to the class of regular continua. Since graph-like continua are regular \cite[Lemma~7]{EGP}, its critical corollary is that in order to check whether a graph-like continuum is $n$-ac, it is sufficient to assume the points lie on edges. It is convenient also to extend our definitions. Let $X$ be a space and $S$ a subset. Then $(S,X)$ is \emph{$n$-ac} (respectively, \emph{$n$-cc}) if for any choice of at most $n$ points from $S$ there is an arc (resp., simple closed curve) in $X$ containing the specified points.

\begin{lemma}
\label{superlem_wlogpointsonedges}
Let $X$ be a regular continuum, $D \subseteq X$ an arbitrary dense subset of $X$, and $n \in \N$. Then $X$ is $n$-ac \textnormal{[}$n$-cc\textnormal{]} if and only if $(D,X)$ is $n$-ac \textnormal{[}$n$-cc\textnormal{]}.
\end{lemma}
\begin{proof}
Only the backwards implication requires proof. Assume that $(D,G)$ is $n$-ac and let $x_0, x_1, \ldots, x_n \in X$ be arbitrary (with $n \geq 1$). Since $X$ is regular, there are open neighbourhoods $U_i \ni x_i$ such that 
\begin{itemize}
\item $\overline{U_i} \cap \overline{U_j} = \emptyset$ for all $0 \leq i < j  \leq n $, and such that 
\item $| \partial U_i | = k_i \in \N$ is minimal with respect to all open neighbourhoods $V$ of $x_i$ with $V \subseteq U_i$ for all $i$.
\end{itemize}
Pick points $y_i \in U_i \cap D$. By assumption, there is an arc [closed curve] $\alpha$ going through $y_0, y_1, \ldots, y_n$, having two of these points as its endpoints. We are now going to argue that we can modify $\alpha$ inside each $\overline{U_i}$ as so to pick up $x_i$ but still remain an arc [closed curve] in $X$. It suffices to give this argument for $i=0$, so write $x=x_0$, $U=U_0$ and $k=k_0$.

Let us assume that $\partial U = \{u_1, \ldots, u_k\}$. Without loss of generality, $\alpha$ passes through $u_1, \ldots, u_i$ in the given linear [cyclic] order (for $1 \leq i \leq k$), and doesn't use $u_{i+1}, \ldots, u_k$. If $k = 1$, it is clear how to use local arc-connectedness of $X$ to add $x_1$ to our arc $\alpha$ [in the $n$-cc case, $k=1$ cannot occur]. Otherwise, since at least one of the endpoints of $\alpha$ lies outside of $U$ [and trivially in the $n$-cc case], we see that $\overline{U} \cap \alpha$ consists of at most $i-1 \leq k-1$ connected arcs (and at least one, as $y_0 \in \overline{U} \cap \alpha$). 

Next, by the fact that $| \partial U | = k \in \N$ was minimal with respect to all neighbourhoods of $x$ contained in $U$, it follows from Menger's $n$-od Theorem that there is a $k$-fan $F$ with center $x$ and leaves in $\alpha$ contained in $\overline{U}$, see \cite{menger} or \cite{nobling}. By the pigeon hole principle, two leaves of the fan $F$ must lie on the same connected component of $\overline{U} \cap \alpha$, and so it is clear how to include $x$ into our arc [closed curve] $\alpha$. As this procedure can be repeated for all $i = 1, \ldots, n$, the proof is complete.
\end{proof}

\subsection{Freudenthal compactification of locally finite connected graphs}

In the proof of the next theorem, we need the following standard lemma saying that the number of edges in a graph leaving a certain vertex set is \emph{submodular}. For a subset $A \subset V(G)$ write $\partial A = E(A,V \setminus A)$ for the induced edge cut, and $A^\complement$ for $V(G) \setminus A$.

\begin{lemma}
\label{l_cutlemma}
Let $G$ be a graph, $A,A' \subset V(G)$. Then
\[
\cardinality{\partial A} + \cardinality{\partial A'} \geq \max \Set{\cardinality{\partial \p{A\cap A'}}+\cardinality{\partial \p{A \cup A'}}, \cardinality{\partial \p{A\setminus A'}}+\cardinality{\partial \p{A' \setminus A}}}. 
\]
\end{lemma}

\begin{proof}
We indicate the short argument of this folklore lemma: We have to verify that every edge $e$ that is counted on the right will also be counted on the left, and if it is counted say in both $\partial \p{A\cap A'}$ and $\partial \p{A \cup A'}$ on the right, it is also counted in both sums on the left.

If $e \in \partial \p{A\cap A'}$, then e joins a vertex $v \in A\cap A'$ to a vertex $w$ that fails to lie in $A$ or which fails to lie in $A'$. In the first case, $e \in \partial A$, and in the second case we have $e \in \partial A'$. Since $\partial \p{A \cup A'} = \partial \p{A^\complement \cap A'^\complement}$, the same holds for edges in $\partial \p{A \cup A'}$: every such edge lies in $\partial \p{A^\complement} = \partial A$ or in $\partial \p{A'^\complement} = \partial A'$.

Finally, if $e$ is counted twice on the left, i.e., if $e \in \partial \p{A\cap A'}$ and $e \in \partial \p{A \cup A'} = \partial \p{A^\complement \cap A'^\complement}$, then $e$ joins a vertex $v \in A \cap A'$ to some other
vertex, and it also joins some $w \in A^\complement \cap A'^\complement$ to some other vertex. As $A \cap A'$ and
$A^\complement \cap A'^\complement$ are disjoint, we have $e = vw$. But this means that $e \in \partial A$ as well as $e \in \partial A'$, so $e$ is counted twice also on the left.

The other inequality, $\cardinality{\partial A} + \cardinality{\partial A'} \geq \cardinality{\partial \p{A\setminus A'}}+\cardinality{\partial \p{A' \setminus A}}$, now follows from the first one by applying the fact that $\cardinality{\partial B} = \cardinality{\partial \p{B^\complement}}$. 
\end{proof}

The final ingredient for our key Theorem~\ref{lem_wlogpointsonedges2Freudenthal} is an alternative, and more explicit, description of the Freudenthal compactification of a locally finite graph in terms of ends.

Let $G$ be a locally finite connected graph. A $1$-way infinite path is called a \emph{ray}, a $2$-way infinite path is a \emph{double ray}. Two rays $R$ and $S$ in $G$ are \emph{equivalent} if no finite set of vertices separates them. Alternatively, we may say that $G$ contains infinitely many disjoint $R-S$-paths. The corresponding equivalence classes of rays are the \emph{ends} of $G$. The set of ends of a graph $G$ is denoted by $\Omega=\Omega(G)$.

Recall that topologically, we view $G$ as a cell complex with the usual 1-complex topology. Adding its ends compactifies it, with the topology on $G \cup \Omega$ generated by the open sets of $G$ and neighbourhood bases for ends $\omega \in \Omega$ defined as follows: Given any finite subset $S$ of $V(G)$, let $C(S,\omega)$ denote the unique component of $G-S$ that contains a cofinal tail of some (and hence every) ray in $\omega$, and let $\hat{C}(S,\omega)$ denote the union of $C(S,\omega)$ together with all ends of $G$ with a ray in $C(S,\omega)$. As our neighbourhood basis for $\omega$ we take all sets of the form $\hat{C}(S, \omega) \cup \mathring{E}\p{S,C(S,\omega)}$, 
where $S$ ranges over the finite subsets of $V(G)$ and $\mathring{E}(S,C(S,\omega))$ denotes the interior of the edges with one endpoint in $S$ and the other in $C(S,\omega)$. Note that in this topology, we have $\closure{C(S,\omega)} \cap \Omega =\hat{C}(S, \omega) \cap \Omega$.

It is well known that this process of adding the ends does indeed yield the Freudenthal compactification, i.e.\ $\Fr{G} = G \cup \Omega$. In particular it is locally connected at ends, and has neighbourhoods which restrict to zero-dimensional sets on the end space. For further details and proofs see Chapter~8 of \cite{Diestel}.

\begin{theorem}
\label{lem_wlogpointsonedges2Freudenthal}
For the Freudenthal compactification  $\Fr{G}$ of a locally finite connected graph $G$ the following are equivalent for each $n \in \N$:

 (1) $\Fr{G}$ is $n$-ac, \ 
 (2) $(G,\Fr{G})$ is $n$-ac, \ and
 (3) $G$ is $n$-ac.
\end{theorem}

\begin{proof}
The equivalence $(1) \Leftrightarrow (2)$ is a special instance of Lemma~\ref{superlem_wlogpointsonedges}. The implication $(3) \Rightarrow (2)$ is trivial. For $(2) \Rightarrow (3)$ consider $n$ points $x_1, \ldots, x_n \in G$ and find, by assumption, an arc $\alpha$ in $\Fr{G}$ going through the specified points. Our task is to modify this arc $\alpha$ so that it still contains $x_1,\ldots, x_n$ but does not use ends of $G$ anymore.  

Without loss of generality we may assume that start- and end-point of $\alpha$  are amongst the $x_i$. Then it follows from \cite[Prop.~3]{euler} that every end $\omega \in \alpha \cap (\Fr{G} \setminus G)$ has degree $2$ in $\alpha$, meaning that for every finite set of vertices $S \subset V(G)$ there is a bipartition $(A_\omega,B_\omega)$ of $V(G)$ such that:
(i) the induced subgraph $G[A_\omega]$ is connected,
(ii)  $\omega \in \overline{A}$,
(iii) $S \subset B_\omega$, and
(iv) $|E(\alpha) \cap \partial A_\omega| = 2$ (i.e. the arc $\alpha$ uses precisely two edges from the edge cut $E(A_\omega,B_\omega)$). 

Let us call such a set $A_\omega$ with $|E(\alpha) \cap \partial A_\omega| = 2$ a \emph{$2$-neighbourhood} of $\omega$. Moreover, note that $|E(\alpha) \cap \partial A| \geq 2$ whenever $\omega \in \overline{A}$ and $A \subseteq A_\omega$ \quad $(\star)$.
Next, let $S = \{x_1, \ldots, x_n\}$ and choose for every end $\omega \in \alpha \cap (\Fr{G} \setminus G)$ a bipartition $(A_\omega,B_\omega)$ with the above four properties. Since $\alpha \cap (\Fr{G} \setminus G)$ is compact, there are finitely many ends $\omega_1, \ldots, \omega_\ell$ such that $\alpha \cap (\Fr{G} \setminus G) \subseteq \overline{A_{\omega_1}} \cup \cdots \cup \overline{A_{\omega_\ell}}$. 
We may assume that this cover is minimal, i.e.\ for every $i \leq \ell$ there is an end $\epsilon_i \in \alpha \cap (\Fr{G} \setminus G)$ such that  
$\epsilon_i \in \overline{A_i} \setminus \bigcup \{\overline{A_j} : j \ne i\}$ \quad $(\star \star)$. 

\smallskip

\textbf{Claim: }\emph{Every minimal cover of $\alpha \cap (\Fr{G} \setminus G)$ consisting of $2$-neighbourhoods has a disjoint refinement 
consisting of $2$-neighbourhoods. }

The proof of the claim is via induction on the size of the cover. Let us make the convention that $\partial_{\alpha} A := E(\alpha) \cap \partial A$ consists of those boundary edges of $A$ that are used by $\alpha$. If the cover consists of a single element only, there is nothing to show. So we may assume $ \ell \geq 2$ and consider our cover $\Set{A_1, \ldots, A_\ell}$.
Let $\tilde{A}_1:=A_1$ and $\tilde{A}_i := A_i \setminus A_1$ for all $1 < i \leq \ell$. From $(\star)$ and $(\star \star)$ it follows that $\cardinality{\partial_{\alpha} \tilde{A}_i} \geq 2$ for all $i \leq \ell$. 

We shall use Lemma~\ref{l_cutlemma} to see that $\cardinality{\partial_{\alpha} \tilde{A}_i} \leq 2$ for all $i \leq \ell$ as well. This is clear for $\tilde{A}_1$. For $i \geq 2$, Lemma~\ref{l_cutlemma} applied to the graph $(V,E(\alpha))$ implies
\[
4 = \cardinality{\partial_{\alpha} A_1} + \cardinality{\partial_{\alpha} A_i} \geq \cardinality{\partial_{\alpha} \p{A_1\setminus A_i}}+\cardinality{\partial_{\alpha} \p{A_i \setminus A_1}} \geq 2 +\cardinality{\partial_{\alpha} \tilde{A}_i},
\]
where $\partial_{\alpha} \p{A_1\setminus A_i} \geq 2$ follows again from $(\star)$ and $(\star \star)$. Thus, we have $\cardinality{\partial_{\alpha} \tilde{A}_i} = 2$ for all $i \leq \ell$. Applying the induction assumption to the collection $\Set{\tilde{A}_2,\ldots, \tilde{A_\ell} }$ we obtain a disjoint refinement of $2$-neighbourhoods, which together with $A_1$ forms the desired refinement of our original collection. This establishes the claim.

\smallskip

Next, we argue that for each $\tilde{A}_i$, there is a finite edge path $P_i$ in $G[\tilde{A}_i]$ from one edge in $\partial_\alpha \tilde{A}_i$ to the other. Let $\alpha_i \subset \alpha$ be the subarc of $\alpha$ that lies in the closure of $\tilde{A}_i$ in $\Fr{G}$. By definition of the topology of the Freudenthal compactification, for every end $\omega$ in $\alpha_i$, there is a finite subset $T \subset V(G)$ such that $C(T,\omega) \subset \tilde{A}_i$. By compactness, finitely many such $C(T_j,\omega_j)$ for $j \leq N$ say cover the ends used by $\alpha_i$. Now since every $C(T_j,\omega_j)$ is by definition a connected graph, we may recursively in $j$ find a finite edge-path in $C(T_j,\omega_j)$ connecting the first and last point of $\alpha_i \cap C(T_j,\omega_j)$. By doing so, we obtain a finite edge-walk in $G[\tilde{A}_i]$ from one edge in $\partial_\alpha \tilde{A}_i$ to the other, which includes the desired finite edge path $P_i$.

But now we are done: for each $i \leq \ell$, replace $\alpha_i$ by $P_i$. Since each replacement took place in the disjoint subsets $\tilde{A}_i$, this gives rise to an arc completely inside the graph $G$ containing all $n$ points $x_1, \ldots, x_n$ as desired.
\end{proof}

\subsection{Graph-like compactification of locally finite connected graphs}

Since every $7$-ac graph is one, up to homeomorphism, of a finite family, we easily deduce from Theorem~\ref{lem_wlogpointsonedges2Freudenthal} that the Freudenthal compactification of a locally finite graph is $7$-ac only in very limited cases. However, this holds for arbitrary graph-like compactifications (i.e.\ for compactifications with zero-dimensional remainders). 

\begin{prop}\label{7acFreud} Let $G$  be a countable, locally finite graph. Let $\gamma G$ be a graph-like compactification of $G$. 

If $\gamma G$ is $7$-ac then $\gamma G$ is (homeomorphic to) a finite graph (and is one of the $6$ finite graphs which are $7$-ac, or equivalently $\omega$-ac).
\end{prop}
\begin{proof} 
The proof of Theorem~2.12 of \cite{acpaper} shows that the graph $G$ can have at most two vertices of degree $3$ or higher.
If all vertices have degree two, then as above $\gamma G$ is either an arc or a circle. 
If all vertices have degree no more than $2$, but not all are degree $2$, then $G$ is either a finite chain, or an infinite one-way chain. In either case $\gamma G$ is an arc or a circle.
Otherwise, extending from the (at most two) vertices of degree at least $3$, there will be a finite family of: (finite) cycles, finite chains or  infinite one-way chains. The infinite chains have either one or two endpoints in $\gamma G$. In all scenarios, $\gamma G$ is homeomorphic to a finite graph.
\end{proof}

Although Theorem~\ref{lem_wlogpointsonedges2Freudenthal}, as stated, only applies to $n$-arc connectedness, and not $n$-circle connectedness, the $n$-cc property is completely dealt with via the next two lemmas. Indeed, as in the previous result, these apply to arbitrary graph-like compactifications of locally finite graphs.

\begin{lemma}\label{2ccFreud} 
 Let $\gamma G$ be a graph-like compactification of a countable, locally finite graph $G$. Then the following are equivalent: (a) $\gamma G$ is $2$-cc, (b) $\gamma G$ has no cut points, (c) $G$ has no cut points, (d) $G$ is $2$-cc, and (e) $G$ is cyclically connected.
\end{lemma}
\begin{proof}
Since $\gamma G$ is graph-like, the equivalence of (a) and (b) follows from Proposition~\ref{gl-2cc} below. Since no point of the remainder, $\gamma G \setminus G$, can be a cut point of $\gamma G$; while every cut point of $G$ is a cut point of $\gamma G$, we see that (b) and (c) are equivalent. Finally, the characterization of $2$-cc  graphs (Theorem~3.1.1 of \cite{GMP}) yields the remaining equivalences.
\end{proof}

\begin{lemma}\label{3ccFreud} Let $\gamma G$ be a graph-like compactification of a countable, locally finite graph $G$. Then the following are equivalent: (a) $\gamma G$ is $3$-cc, (b) $\gamma G$ is a circle, and (c) $G$ is either a cycle, or a double ray and $\gamma G$ is its one-point compactification.
\end{lemma}
\begin{proof} Suppose $\gamma G$ is $3$-cc. 
The corresponding argument for finite graphs shows that every vertex of $G$ has degree $2$. So $G$ is either a finite cycle, or a double ray. In the latter case, there are only two different graph-like compactifications: $\gamma G$ is either a circle, or an arc -- but in the latter case, $\gamma G$ is not $3$-cc.
\end{proof}

However, Theorem~\ref{lem_wlogpointsonedges2Freudenthal}, stating that a locally finite, countable graph $G$ is $n$-ac if and only if $\Fr{G}$ is $n$-ac, does not extend to general graph-like compactifications for $n \le 6$.

\begin{exam}\label{onepointExs} \ 

\noindent (a) The infinite ladder, $D$, is $5$-ac but not $6$-ac, while $\alpha D$ is $6$-ac.

\noindent (b) The graph $C$ below is $4$-ac but not $5$-ac, while its one-point compactification, $\alpha C$, is $6$-ac.

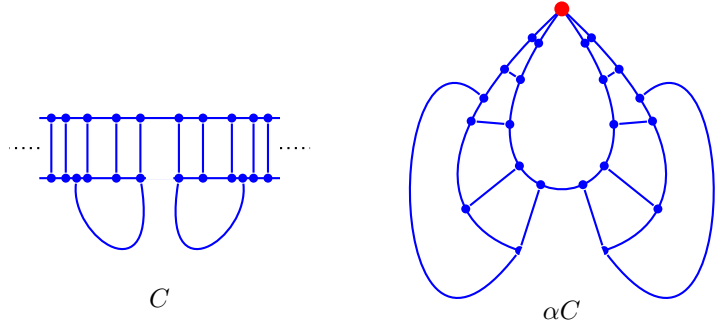
\begin{minipage}{.5\textwidth}
\begin{center}
\begin{tikzpicture}[thick,scale=.8]
\draw[dotted] (-2,-0.5) -- (-2.5,-0.5);
\draw[dotted] (2,-0.5) -- (2.5,-0.5);

\draw[blue] (-2,0) -- (2,0)
node[circle,pos=0.05,fill=blue,inner sep=1.2] (p1) {}
node[circle,pos=0.11,fill=blue,inner sep=1.2] (p2) {}
node[circle,pos=0.2,fill=blue,inner sep=1.2] (p3) {}
node[circle,pos=0.32,fill=blue,inner sep=1.2] (p4) {}
node[circle,pos=0.42,fill=blue,inner sep=1.2] (p5) {}
node[circle,pos=0.58,fill=blue,inner sep=1.2] (p51) {}
node[circle,pos=0.68,fill=blue,inner sep=1.2] (p6) {}
node[circle,pos=0.8,fill=blue,inner sep=1.2] (p7) {}
node[circle,pos=0.89,fill=blue,inner sep=1.2] (p8) {}
node[circle,pos=0.95,fill=blue,inner sep=1.2] (p9) {};

\draw[blue] (-2,-1) -- (2,-1)
node[circle,pos=0.05,fill=blue,inner sep=1.2] (n1) {}
node[circle,pos=0.11,fill=blue,inner sep=1.2] (n2) {}
node[circle,pos=0.155,fill=blue,inner sep=1.2] (n25) {}
node[circle,pos=0.2,fill=blue,inner sep=1.2] (n3) {}
node[circle,pos=0.32,fill=blue,inner sep=1.2] (n4) {}
node[circle,pos=0.42,fill=blue,inner sep=1.2] (n5) {}
node[circle,pos=0.58,fill=blue,inner sep=1.2] (n51) {}
node[circle,pos=0.68, fill=blue,inner sep=1.2] (n6) {}
node[circle,pos=0.8,fill=blue,inner sep=1.2] (n7) {}
node[circle,pos=0.845,fill=blue,inner sep=1.2] (n75) {}
node[circle,pos=0.89,fill=blue,inner sep=1.2] (n8) {}
node[circle,pos=0.95,fill=blue,inner sep=1.2] (n9) {};

\draw[blue] (p1) -- (n1);
\draw[blue] (p2) -- (n2);
\draw[blue] (p3) -- (n3);
\draw[blue] (p4) -- (n4);
\draw[blue] (p5) -- (n5);
\draw[blue] (p51) -- (n51);
\draw[blue] (p6) -- (n6);
\draw[blue] (p7) -- (n7);
\draw[blue] (p8) -- (n8);
\draw[blue] (p9) -- (n9);

\draw[white,very thick] (n5)--(n51);

\draw[blue] (n5) .. controls (0,-3) and (-1.5,-2) .. (n25);
\draw[blue] (n51) .. controls (0,-3) and (1.5,-2) .. (n75);

\node at (0,-3) {$C$};
\end{tikzpicture}
\end{center}
\end{minipage}
\begin{minipage}{.5\textwidth}
\begin{center}
\hspace*{-80pt}
\begin{tikzpicture}[thick,scale=.8]
\draw[blue] (1,0) .. controls (-2,-4) and (4,-4) .. (1,0)
node[circle,pos=0.05,fill=blue,inner sep=1.2] (p1) {}
node[circle,pos=0.11,fill=blue,inner sep=1.2] (p2) {}
node[circle,pos=0.2,fill=blue,inner sep=1.2] (p3) {}
node[circle,pos=0.32,fill=blue,inner sep=1.2] (p4) {}
node[circle,pos=0.42,fill=blue,inner sep=1.2] (p5) {}
node[circle,pos=0.58,fill=blue,inner sep=1.2] (p51) {}
node[circle,pos=0.68,fill=blue,inner sep=1.2] (p6) {}
node[circle,pos=0.8,fill=blue,inner sep=1.2] (p7) {}
node[circle,pos=0.89,fill=blue,inner sep=1.2] (p8) {}
node[circle,pos=0.95,fill=blue,inner sep=1.2] (p9) {};

\draw[blue] (1,0) .. controls (-5,-5.5) and (7,-5.5) .. (1,0)
node[circle,pos=0.03,fill=blue,inner sep=1.2] (n1) {}
node[circle,pos=0.065,fill=blue,inner sep=1.2] (n2) {}
node[circle,pos=0.1,fill=blue,inner sep=1.2] (n25) {}
node[circle,pos=0.13,fill=blue,inner sep=1.2] (n3) {}
node[circle,pos=0.28,fill=blue,inner sep=1.2] (n4) {}
node[circle,pos=0.42,fill=blue,inner sep=1.2] (n5) {}
node[circle,pos=0.58,fill=blue,inner sep=1.2] (n51) {}
node[circle,pos=0.72, fill=blue,inner sep=1.2] (n6) {}
node[circle,pos=0.87,fill=blue,inner sep=1.2] (n7) {}
node[circle,pos=0.9,fill=blue,inner sep=1.2] (n75) {}
node[circle,pos=0.935,fill=blue,inner sep=1.2] (n8) {}
node[circle,pos=0.97,fill=blue,inner sep=1.2] (n9) {};

\draw[blue] (p1) -- (n1);
\draw[blue] (p2) -- (n2);
\draw[blue] (p3) -- (n3);
\draw[blue] (p4) -- (n4);
\draw[blue] (p5) -- (n5);
\draw[blue] (p51) -- (n51);
\draw[blue] (p6) -- (n6);
\draw[blue] (p7) -- (n7);
\draw[blue] (p8) -- (n8);
\draw[blue] (p9) -- (n9);
\draw (1,0) node[circle, fill=red, inner sep=2] (v1){};
\draw[white,fill=white] (1,-4.3) circle (.72);

\draw[blue] (n5) .. controls (-2,-7) and (-2,0) .. (n25);
\draw[blue] (n51) .. controls (4,-7) and (4,0) .. (n75);
\node at (1,-5) {$\alpha C$};
\end{tikzpicture}
\end{center}
\end{minipage}

\vspace*{-35pt}

\end{exam}

\noindent\emph{Proof.}
\emph{For (a):}
Let $D$ be the usual double ladder, i.e. $V(D) = \Set{0,1} \times \Z$ in which two vertices $(m,n)$ and $(m',n')$ are adjacent if and only if $|m-m'|+ |n-n'|=1$. 
Using the characterizations from \cite{GMP}, it follows that $D$ is $5$-ac but not $6$-ac.

We focus on showing $\alpha D$ is $6$-ac. 
Since we may assume our six points $x_1,\ldots,x_6$  lie on edges, we may find $n\geq 5$ large enough such that 
$x_1,\ldots, x_6 \in D\left[\Set{0,1} \times [-n,n]\right]$. 

Set $G_1=D\left[\Set{0,1} \times [-n,n]\right]$. Take a disjoint copy of $G_1$, and  modify it to form a graph $G_2$ as follows: first, remove the edge corresponding to $\Set{(0,0),(0,1)}$, and second, subdivide the edges $\Set{(0,-2),(0,-3)}$ and $\Set{(0,3),(0,4)}$ by vertices $a$ and $b$, and, finally, add new edges from $a$ to $(0,0)$ and $(0,1)$ to $b$. 
Let us write $e = \Set{(1,0),(1,1)}$ for the unique bridge of $G_2$, and $G_2^+:=G_2[\Set{0,1} \times \Set{1,\ldots,n}]$ and $G_2^-:=G_2[\Set{0,1} \times \Set{0,-1,\ldots,-n}]$ for the two components of $G_2 - e$.

Now consider the auxiliary graph $G= G_1 \sqcup G_2$ where we additionally add four new edges:
(1) $f^+$ between the copies of $(0,n)$,
(2) $f^-$ between the copies of $(0,-n)$, 
(3) $g^+$ between the copies of $(1,n)$, and (4) $g^-$ between the copies of $(1,-n)$.

\begin{wrapfigure}{r}{5.5cm}
\begin{tikzpicture}[thick]

\draw[cyan] (105:1.35) arc [radius=1.35,start angle=105, delta angle=60];
\draw[cyan] (75:1.35) arc [radius=1.35,start angle=75, delta angle=-60];
\draw[cyan] (0,2) arc [radius=2,start angle=90, delta angle=75];
\draw[cyan] (0,2) arc [radius=2,start angle=90, delta angle=-75];

\draw (0,-1.35) arc [radius=1.35,start angle=-90, delta angle=90];
\draw (0,-1.35) arc [radius=1.35,start angle=-90, delta angle=-90];
\draw (0,-2) arc [radius=2,start angle=-90, delta angle=90];
\draw (0,-2) arc [radius=2,start angle=-90, delta angle=-90];

\draw[blue] (0:1.35) arc [radius=1.35,start angle=0, delta angle=15];
\draw[blue] (180:1.35) arc [radius=1.35,start angle=180, delta angle=-15];
\draw[red] (0:2) arc [radius=2,start angle=0, delta angle=15];
\draw[red] (180:2) arc [radius=2,start angle=180, delta angle=-15];

\foreach \x in {15,30,...,75}
{
    \node[circle,fill=cyan,inner sep=1.2] (p) at (\x:2) {};
    \node[circle,fill=cyan,inner sep=1.2] (p) at (\x:1.35) {};
    \draw[cyan] (\x:2) -- (\x:1.35);
}

\foreach \x in {105,120,...,165}
{
    \node[circle,fill=cyan,inner sep=1.2] (p) at (\x:2) {};
    \node[circle,fill=cyan,inner sep=1.2] (p) at (\x:1.35) {};
    \draw[cyan] (\x:2) -- (\x:1.35);
}

\foreach \x in {0,-15,...,-180}
{
    \node[circle,fill=black,inner sep=1.2] (p) at (\x:2) {};
    \node[circle,fill=black,inner sep=1.2] (p) at (\x:1.35) {};
    \draw[black] (\x:2) -- (\x:1.35);
}

\draw[cyan] (105:1.35) .. controls (0,-0.01) and (-0.01,-0.01) .. (142.5:1.35);
\node[circle,fill=cyan,inner sep=1.2] at (142.5:1.35) {};

\draw[cyan] (75:1.35) .. controls (0,-0.01) and (-0.01,-0.01) .. (37.5:1.35);
\node[circle,fill=cyan,inner sep=1.2] at (37.5:1.35) {};

\node at (0,1.8) {$e$};

\node at (1.4,-2.1) {$G_1$};

\node at (1.8,1.6) {$G_2^+$};
\node at (-1.8,1.6) {$G_2^-$};

\node[blue] at (1,0.1) {$g^+$};
\node[red] at (2.4,0.2) {$f^+$};

\node[blue] at (-0.9,0.1) {$g^-$};
\node[red] at (-2.3,0.2) {$f^-$};

\end{tikzpicture}
\end{wrapfigure}

It follows from \cite[Theorem~3.4.1]{GMP} that $G$ is $6$-ac, and so there is an arc $\alpha$ in $G$ containing $x_1,\ldots,x_6$ and, without loss of generality, starting and ending in points $x_i \neq x_j$. In particular, $\alpha$ starts and ends outside of $G_2$. Moreover, note that $\partial_G G_2^+ = \Set{e,f^+,g^+}$ is a $3$-edge cut, and so if $\alpha$ contains points from $G_2^+ $ then $\alpha$ will cross this cut in precisely two edges, and so $\beta^+ = \alpha \cap G_2^+$ will be a subarc of $\alpha$. Similarly, $\beta^- = \alpha \cap G_2^-$ will be a subarc of $\alpha$. But then it is clear that by replacing $\beta^+$ and $\beta^-$ with suitable arcs in the corresponding connected components of $\alpha D \setminus G_1$ (where say an $e-f^+$ arc will be replaced by an $\infty-f^+$-arc in $\alpha D$), we may lift $\alpha$ to an arc in $\alpha D$  witnessing $6$-ac. 

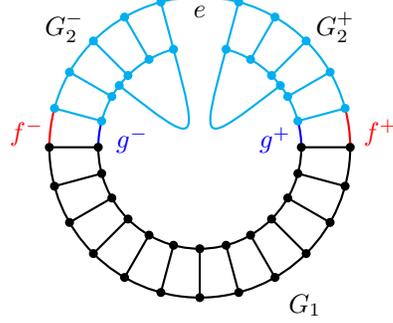
\emph{For (b):} That $\alpha C$ is $6$-ac can be directly checked by a case-by-case analysis. 

To see that $C$ is $4$-ac but not $5$-ac we can apply the characterizations of \cite{GMP} as follows. First note that removing the middle edge disconnects $C$ into two components $C_+, C_-$ which are isomorphic. Since $C_\pm$ is cyclically connected, and no two vertices cut it into $4$ or more components, it is  $4$-ac by \cite[Theorem~3.2.1]{GMP}. As $C$  
is $3$-regular it follows from \cite[Theorem~3.2.3]{GMP} that $C$ is $4$-ac. On the other hand, since removing the middle edge disconnects $C$, it is not cyclically connected. 
Now \cite[Theorem~3.3.1]{GMP} states that for  $C$ to be $5$-ac it must be homeomorphic to one of: an arc, ray, double ray, lollipop with or without end point, dumbbell or figure-eight, and it is clearly not homeomorphic to any of these spaces. 
\hfill $\qed$

\bigskip

The argument given that $\alpha D$ is $6$-ac is straightforward, but follows from an \emph{ad hoc} reduction to the combinatorial graph characterization of $6$-ac. The direct check that $\alpha C$ is $6$-ac is lengthy and tedious, in sharp contrast to the simple arguments, from the combinatorial characterizations, that $C$ is $4$-ac but not $5$-ac. 
These two examples demonstrate some of the difficulties in determining when a graph-like compactification of a locally finite, connected graph $G$ is $n$-ac, and also the value in having a combinatorial characterization.

\begin{prob}
Find a combinatorial characterisation in the spirit of the results for (infinite) 1-complexes in \cite{GMP} for when a graph-like compactification of a locally finite, connected graph $G$ is $n$-ac.
\end{prob}

A place  to start would be to discover when the one-point compactification of a graph is $6$-ac. 
\hideme{\marginpar{Max: There are of course examples where $\alpha A$ is even7-ac and hence $\omega$-ac. Should exclude these...
Consider a graph $G$ consisting three rays / half open intervals glued together at their endpoints. Then $G$ has a 3-cut point, so is not $3$-ac by \cite[2.3.3]{GMP}. However, $\alpha G$ is a $\theta$-curve, so even $7$-ac.

Paul: yes... I was looking for simple natural questions. So I have removed the specific questions.}
}

\section{General Graph-like Continua}
\label{sec_3}

In this section we first develop some machinery for graph-like spaces with the aim of connecting them, via inverse limits with `nice'  bonding maps, to finite graphs. This machinery then yields tests for a graph-like continuum to be, or not to be, $n$-ac or $n$-cc. In  Proposition~\ref{gl-2cc} these tests are refined to characterize $2$-cc graph-like continua. In the next section our machinery and tests for graph-likes are applied to construct various examples.

\subsection{Graph-like spaces as inverse limits} \label{mach_gl}
Here we develop techniques of Espinoza and the present authors in \cite{EGP}, to detect when a continuum is graph-like, and characterize when a graph-like continuum is Eulerian.

For convenience let us say that a map $\pi$ from one graph-like continuum, $X$, to another, $Y$, is \emph{nice} if it is surjective, monotone (fibres, $\pi^{-1} \{v\}$, are connected) and maps vertices to vertices, and edges either homeomorphically to another edge, or to a vertex.

Let $X$ be a graph-like continuum with vertex set $V$. By subdividing edges once, if necessary, we may assume that every edge of $X$ has two distinct endpoints in $V$, i.e.\ that the graph-like continuum is \emph{simple}.

For a clopen subsets $U,U' \subseteq V$, not necessarily different, $E(U,U')$ denotes the set of edges with one endpoint in $U$ and the other endpoint in $U'$. It is not hard to see, \cite[Lemma~1]{EGP}, that $E(U,V \setminus U)$ is always finite. A \emph{multi-cut} is a partition $\mathcal{U}=\{U_1, U_2, \ldots , U_n\}$ of $V$ into pairwise disjoint clopen sets such that for each $i$, the \emph{induced subspace} $X[U_i]$ of $X$, i.e.\ the closed graph-like subspace with vertex set $U_i$ and edge set $E(U_i,U_i)$, is connected. The \emph{multigraph associated with $\mathcal{U}$} is the quotient $G_X(\mathcal{U})=G(\mathcal{U})=X / \{X[U] : U \in \mathcal{U}\}$.
Let $p_\mathcal{U} \colon X \to G(\mathcal{U})$ denote the quotient mapping from $X$ to the multigraph associated with $\mathcal{U}$. 
We note that $G(\mathcal{U})$ is indeed a finite, connected multi-graph, and that $p_\mathcal{U}$ is nice.
Conversely, if $p$ is a nice map of $X$ to a finite, connected graph $G$, then there is a multi-cut $\mathcal{U}$ such that $G=G(\mathcal{U})$ and $p_\mathcal{U}$ \emph{realizes} $p$ in the sense that they are identical on the vertices of $X$, and they carry the same edges of $X$ to the same edges of $G$.

A sequence, $(\mathcal{U}_n)_n$, of multi-cuts of $X$ is  \emph{cofinal} if  for every multi-cut $\mathcal{U}$ there is an $\mathcal{U}_n$ which refines it. 
According to Theorem~13 of \cite{EGP}, for any cofinal sequence, $(\mathcal{U}_n)_n$, of multi-cuts, the graph-like continuum $X$ is naturally homeomorphic to an inverse limit $\varprojlim G_X(\mathcal{U}_n)$, where the bonding maps are all nice. 
Conversely, if a space $X$ is homeomorphic to an inverse limit, $\varprojlim G_n$, where the $G_n$ are finite, connected graphs, and all bonding maps are nice, then (Theorem~14 of \cite{EGP}) $X$ is a graph-like continuum. Note that in this case, for every $m$, the projection map, typically denoted, $p_m$, from $\varprojlim G_n$ to $G_m$ is nice, and so is realized as a $p_{\mathcal{U}_m}$ for some multi-cut $\mathcal{U}_m$.

\subsection{Sufficient conditions}

The following lemma -- a special case of Lemma~\ref{superlem_wlogpointsonedges} -- records that as in the case with graphs, also for graph-like continua we may choose our points $x_1, \ldots, x_n$ without loss of generality to be interior points of edges.
\begin{lemma}
\label{lem_wlogpointsonedges3}
Let $X$ be a graph-like continuum with vertex set $V$. Let $n \in \N$. Then $X$ is $n$-ac \textnormal{[}$n$-cc\textnormal{]} if and only if $(X \setminus V,X)$ is $n$-ac \textnormal{[}$n$-cc\textnormal{]}.
\end{lemma}

\begin{lemma}
\label{lem_lifting}
Let $\script{U}$ be a multi-cut of a graph-like continuum $X$. Then every arc \textnormal{[}simple closed curve\textnormal{]} in $G=G_X(\mathcal{U})$ lifts to an arc \textnormal{[}simple closed curve\textnormal{]} in $X$.
\end{lemma}

\begin{proof}
Since the quotient mapping $p_\mathcal{U} \colon X \to G(\mathcal{U})$ is nice, it follows that for every vertex $v$ of $G$, its fibre $p_\mathcal{U}^{-1}(v) = X[U]$ for some $U \in \script{U}$ is an connected, and hence arc-connected subcontinuum of $X$, see \cite[Lemma~2]{EGP}. Thus, we may lift any arc \textnormal{[}simple closed curve\textnormal{]} $\alpha$ in $G=G_X(\mathcal{U})$ by filling in suitable subarcs inside each fibre $p_\mathcal{U}^{-1}(v) = X[U]$ for every vertex $v \in \alpha$.
\end{proof}

\begin{cor}
Let $X$ be a graph-like continuum.
If $G_X(\mathcal{U})$ is $n$-ac \textnormal{[}$n$-cc\textnormal{]} for every multicut $\script{U}$ of $X$, then $X$ is $n$-ac \textnormal{[}$n$-cc\textnormal{]}.
\end{cor}

\begin{proof}
By Lemma~\ref{lem_wlogpointsonedges3} it suffices to consider points $x_1, \ldots, x_n$ lying on edges of $X$, say $x_i \in e_i$. Since $\varprojlim G_X(\mathcal{U}_n) \cong X$, there is a multicut $\script{U}$ of $X$ such that $e_1, \ldots, e_n$ are all displayed in the finite graph $G=G_X(\mathcal{U})$. By assumption, $G_X(\mathcal{U})$ is $n$-ac \textnormal{[}$n$-cc\textnormal{]}, and so there is an arc \textnormal{[}simple closed curve\textnormal{]} in $G$ containing the distinct points $p_\script{U}(x_1), \ldots, p_\script{U}(x_n)$. The assertion is then immediate by Lemma~\ref{lem_lifting}.
\end{proof}

\subsection{Necessary conditions}

Call a graph $G$ \emph{$n$-E} ($n$-Eulerian) if for every $n$ or fewer points in $G$ there is an edge disjoint closed trail in $G$ containing the points. Equivalently, we may say that every $n$ edges of $G$ lie on a common Eulerian subgraph of $G$. Observe that a finite graph is Eulerian if and only if it is $n$-E for all $n$.  Call a graph $G$ \emph{$n$-oE} ($n$-open Eulerian) if for every $n$ or fewer points in $G$ there is an edge disjoint (possibly not closed) trail in $G$ containing the points.

\begin{prop}\label{necc}
Let $X$ be a graph-like continuum.

(a) If $X$ is $n$-cc, then for every multi-cut $\mathcal{U}$ of $X$ the graph $G(\mathcal{U})$ is $n$-E.

(b) If $X$ is $n$-ac, then for every multi-cut $\mathcal{U}$ of $X$ the graph $G(\mathcal{U})$ is $n$-oE.
%
%
\end{prop}
\begin{proof} 
We prove (a). So suppose $X$ is $n$-cc. Write $X$ as an inverse limit $X= \varprojlim G_k$ of graphs, with nice bonding maps. We verify that each $G_k$ is $n$-E. 

Fix $k$. Let $p_k$ be the nice projection from $X$ to $G_k$. Take no more than $n$ points from $G_k$, say $x_1, \ldots , x_n$. Pick $y_1, \ldots, y_n$ in $X$, such that $p_k(y_i)=x_i$, for $i=1, \ldots, n$. As $X$ is $n$-cc, there is a simple closed curve $S$ in $X$ containing these points.
The projection of $S$ under $p_k$ into $G_k$ is an edge-disjoint closed trail in $G_k$ which contains all the $x_i$. This shows that $G_k$ is $n$-E.

The proof of (b) is very similar. In place of a circle we get an arc $\alpha$ containing $y_1, \ldots, y_n$. Its projection in $G_k$ is an edge-disjoint trail which may or may not be closed, but definitely contains the points $x_1, \ldots , x_n$. Thus $G_k$ is $n$-oE.
\end{proof}




\subsection{$2$-cc Graph-like Continua}
A space $X$ is \emph{$3$-sac} if given any three points, $x_1, x_2, x_3$ of $X$, there is an arc in $X$ starting at $x_1$, passing through $x_2$, and ending at $x_3$. 
The main result here is the following one showing that in graph-like continua being $2$-cc is equivalent to being $3$-sac, and characterizing these properties in terms of the standard properties of the graph-like continuum and, also, its inverse limit representation.

\begin{prop}\label{gl-2cc} 
For a  graph-like continuum $X$, the following are equivalent:

(1)  $X$ is $3$-sac, (2) $X$ is $2$-cc, 
(3) $X$ has no cut points,


(4) for every representation $X=\varprojlim G_m$, where each $G_m$ is a finite, connected graph and each bonding map is nice, there is a $m$ such that $G_m$ has no cut-point,

(5) $X$ can be represented as $X=\varprojlim G_n$, where each $G_n$ is a finite, $2$-cc graph and each bonding map is nice.
\end{prop}

The next lemma shows the equivalence of (1), (2) and (3) even among all Peano continua. Lemma~\ref{lem_PeanoCut} also shows that (5) is equivalent to (5${}'$) where `$2$-cc' is replaced by `no cut points'. Then
the equivalence of (3), (4) and (5${}'$) is the $k=2$ case of Proposition~\ref{kconn}.

\begin{lemma}
\label{lem_PeanoCut}
For a Peano continuum $X$, the following are equivalent:

(1) $X$ is $3$-sac, (2) $X$ is $2$-cc, and (3) $X$ has no cut points.
\end{lemma}

\begin{proof}
The equivalence of (1) and (2) for any continuum $X$ was established in \cite{sacpaper}, and was shown to be equivalent to (3${}'$)  $X$ is arc connected, has no arc-cut point, and has no arc end points ($x$ is an arc end point if there are not two arcs intersecting only at $x$). Clearly (3${}'$)  implies (3) (cut points are arc-cut points).
 
 Suppose $X$ is Peano. Then it is arc connected. We show if $X$ contains an arc-cut point or an arc end point then it contains a cut point, and so (3) implies  (3${}'$).

First, assume that $X$ has an arc end point $x$. Recall the result by N\"obling \cite{nobling} that if a point $x$ in a Peano continuum $X$ has order at least $n$ (i.e. any small enough neighbourhood of $x$ has boundary at least of size $n$) then $X$ contains an $n$-pod with center $x$, i.e.\ a union of $n$ many arcs with only the point $x$ in common. Thus, an arc end point must necessarily have order $1$, and so we have found many cut-points.

Second, it is not hard to show that every arc-cut point $x$ of a Peano continuum $X$ must necessarily be a cut point. Indeed, suppose $X \setminus \{x\}$ has at least $2$ arc-components. Let $Y$ be an arc-component. Using local connectedness, it is easy to show that $Y$ must be closed in $X \setminus \{x\}$, and further, that the collection $\{Y \subset X \setminus \{ x\} \colon Y \text{arc-component} \}$ is a locally finite collection of sets. Thus, one arc component against the union of the rest is a partition of $X \setminus \{x\}$ into non-empty closed sets. 
\end{proof}


In analogy to graphs, call a graph-like continuum  \emph{$k$-connected} if the deletion of $k-1$ vertices does not disconnect it. Note that a graph-like continuum is $2$-connected if and only if it has no cut points (if removing a point on an edge disconnects, then so does removing either of the end points of the edge).

A \emph{$k$-pre-cutting} is a triple $(Y,A,B)$ where $Y$ is a set of vertices with $|Y| < k$, and $A,B$ are subcontinua with $A \cup B = X$ and $A\cap B = Y$. A \emph{$k$-cutting} of $X$ is a non-trivial $k$-pre-cutting, $(Y,A,B)$ where by \emph{non-trivial} we mean that $A \setminus Y$ and $B \setminus Y$ are non-empty. 
Observe that if $(Y,A,B)$ is a $k$-cutting then $X \setminus Y$ is disconnected. 
Conversely, if $Y$ is a set of vertices of size $<k$, and removing $Y$ from $X$ disconnects $X$, say $X \setminus Y = U \cup V$ where $U$ are disjoint, open and non-empty, then $(Y,A,B)$ is a $k$-cutting, where $A=U \cup Y$ and $B=V\cup Y$.

If $f\colon Z \to W$ is a nice map from $Z$ to another graph-like continuum, $W$, and $(Y,A,B)$ is a $k$-pre-cutting in $Z$, then $(f(Y),f(A),f(B))$ is a $k$-pre-cutting in $W$.

\begin{prop}\label{kconn}
For a graph-like continuum $X$, the following are equivalent:
\begin{enumerate}
\item[(a)] $X$ is $k$-connected,
\item[(b)] for every representation $X=\varprojlim G_m$, where each $G_m$ is a finite, connected graph and each bonding map is nice, there is an $m$ such that $G_m$ is $k$-connected, and
\item[(c)] $X$ can be represented, $X=\varprojlim G_n$, where each $G_n$ is a finite, $k$-connected graph and each bonding map is nice.
\end{enumerate}
\end{prop}

\begin{proof}
Suppose (b) holds. Fix a representation $X=\varprojlim G_m$. For any $n$, we have $X=\varprojlim_{m \ge n} G_m$, so, by (b), for some $m_n \ge n$ we know $G_{m_n}$ is $k$-connected.  Letting $H_n=G_{m_n}$, we have a representation $X=\varprojlim H_n$ where all the graphs involved are $k$-connected. Thus (c) follows from (b).

Next suppose (a) fails, we show (c) also fails, and so (c) implies (a). Fix a $k$-cutting $(Y,A,B)$ of $X$.   Take any  representation $X=\varprojlim G_m$, where each $G_m$ is a connected, finite graph, and each bonding map is nice. Denote, as usual, the projection map of $\varprojlim G_m$ to $G_m$ by $p_m$, and recall it is nice. 
Pick $x$ in $A \setminus Y$, and $ b \in B \setminus Y$. Find $m$ sufficiently large that in $G_m$ the points $p_m(a)$ and $p_m(b)$ are distinct and not contained in $p_m(Y)$.
Then, as $p_m$ is nice, $(p_m(Y), p_m(A), p_m(B))$ is a $k$-cutting of $G_m$, which, therefore, is not $k$-connected.

Finally we show if (b) is false then so is (a). 
Fix a representation $X=\varprojlim G_m$,  where each $G_m$ is a finite, connected graph which is not $k$-connected, and each bonding map, $\pi_m \colon G_{m+1} \to G_m$, is nice.
Let $\mathcal{T}$ be the set of all finite sequences $\langle (Y_1, A_1,B_1), \ldots$,$(Y_n,A_n,B_n)\rangle$ where each $(Y_m,A_m,B_m)$ is a $k$-pre-cutting of $G_m$,  $Y_m = \pi_m(Y_{m+1})$, $A_m = \pi_m(A_{m+1})$, $B_m=\pi_m(B_{m+1})$, and some term in the sequence is non-trivial (i.e. a $k$-cutting, and note all subsequent terms of the sequence are also non-trivial).

Order $\mathcal{T}$ by extension to get a tree. Observe that every sequence in $\mathcal{T}$ has only finitely many immediate successors (indeed there are only finitely many $k$-pre-cuttings, $(Y_m,A_m,B_m)$, of $G_m$, since $Y_m$ is a set of vertices of the finite graph $G_m$).
Further $\mathcal{T}$ is infinite. To see this fix $n$. We show there is a sequence in $\mathcal{T}$ of length $n$. 
Well, by hypothesis, $G_n$ is not $k$-connected, and so contains a $k$-cutting $(Y_n,A_n,B_n)$. Then $\langle (Y_1, A_1,B_1), \ldots, (Y_m,A_m,B_m), \ldots (Y_n,A_n,B_n)\rangle$ is in $\mathcal{T}$ where $(Y_m,A_m,B_m)=(\pi_m(Y_{m+1}),\pi_m(A_{m+1}),\pi_m(B_{m+1}))$, for $m=n-1, \ldots, 1$.

By K\"{o}nig's Lemma, $\mathcal{T}$ has an infinite branch, $\sigma_1, \sigma_2, \ldots, \sigma_m, \ldots$.  So we get an infinite compatible sequence of $k$-pre-cuttings $\langle (Y_m,A_m,B_m) \rangle_m$.
Let $A=\varprojlim A_m$, $B=\varprojlim B_m$ and $Y=A \cap B$. 
Then, by compatibility, $A$ and $B$ are subcontinua of $X$, $X = A \cup B$ and $Y$ is a set of vertices. But some term of the branch is non-trivial, and so from that point on, all the $k$-pre-cuttings are non-trivial. Further the sets $Y_m$ must stabilize. Thus $(Y,A,B)$ is a non-trivial $k$-pre-cutting, and $X$ is not $k$-connected.
\end{proof}

\subsection{Distinguishing graph-like continua}

Let $X$ be a graph-like continuum. For distinct vertices $v$ and $w$ from $X$ define $k_X(v,w)$, the \emph{edge connectivity} between $v$ and $w$, to be the minimal number of edges whose removal separates $v$ and $w$ (i.e. which form an edge-cut between $v,w$). Note that $k_X(v,w)$ is well-defined, and by Menger's theorem for graph-like continua, \cite[Theorem 22]{EGP}, $k=k_X(v,w)$ equals the maximum size of a family of edge-disjoint $v-w$-paths. 

\begin{lemma} \label{ec_nice} Let $X$ be a graph-like continuum containing distinct vertices $v$ and $w$. 
If $Y$ is another graph-like continuum and $f$ is a nice map of $X$ to $Y$ then $k_X(v,w) \le k_Y(f(v),f(w))$ provided $f(v) \ne f(w)$. 
\end{lemma}
\begin{proof}
Pick edges $e_1, \ldots, e_k$ that separate $f(v)$ from $f(w)$ in $Y$. Then, as $f$ is nice, those same edges exist in $X$ and separate $v \in f^{-1} \{f(v)\}$ from $w \in  f^{-1} \{f(w)\}$. 
\end{proof}

\begin{lemma}
\label{lem_homeomorphismequalsisomorphism}
Let $X,X'\neq S^1$ be graph-like continua with standard representations $X=(V,E)$ and $X' = (V',E')$. Then every homeomorphism $f \colon X \to X'$ is a nice isomorphism of graph-like spaces. 
\end{lemma}

\begin{proof}
Since the degree of a point is a topological property, and hence preserved under homeomorphisms, it follows that any homeomorphism $f \colon X \to X'$ must map $V$ homeomorphically to $V'$ and therefore, by considering complements, edges to edges. Since it is bijective, it is trivially monotone.
\end{proof}

In particular, the previous two lemmas allow us to use \emph{combinatorial information} to show that two graph-like continua $X$ and $Z$ are non-homeomorphic. Indeed, it suffices to find distinct $v$ and $w$ in $X$ such that $k_X(v,w) \ne k_Z(v',w')$ for all distinct $v',w'$ in $Z$. This is simplified by the next lemma.

\begin{lemma} \label{edge_c} Let $X=(V,E)$ be a graph-like continuum, with representation $X=\varprojlim G_k$, of connected graphs, with nice bonding maps. Let $v$ and $w$ be distinct vertices of $X$ and define $s=s(v,w)$ to be minimal such that $p_s(v) \ne p_s(w)$.

Then $k_X(v,w) = \min \{ k_{G_t} (p_t(v),p_t(w)) : t \ge s\} =:\underline{k}$.

Further, the sequence $\left(k_{G_t} (p_t(v),p_t(w)) \colon t \ge s \right)$
is decreasing and eventually constant. It stabilizes, so $k_X(v,w)=k_{G_t}(p_t(v),p_t(w))$,
at the minimal $t$ for which there is a set $\mathcal{E}$ of edges in $X$ of size $k_X(v,w)$ separating $v$ and $w$ such that all members of $\mathcal{E}$ exist in $G_t$.
\end{lemma}
\begin{proof}
First note that $\underline{k} \ge k_X(v,w)$ if and only if for all $t \ge s$ we have $k_{G_t} (p_t(v),p_t(w)) \ge k_X(v,w)$. Now, for each $t \ge s$, apply Lemma~\ref{ec_nice} to the nice map $p_t \colon X \to G_t$. 

Conversely, note $k_X(v,w) \ge \underline{k}$ if and only if for some $t \ge s$ we have $k_X(v,w) \ge k_{G_t} (p_t(v),p_t(w))$.
Fix open edges $e_1, \ldots, e_k$ in $X$ separating $v$ from $w$. Specifically, say $v$ is in $C$, $w$ is in $D$, where $C,D$ form of a clopen partition of $X \setminus \bigcup_i e_i$. Pick $t$ sufficiently large that $t \ge s$ and $p_t$ is a homeomorphism on each of the fixed edges (so, we can suppose $e_1, \ldots, e_k$ are edges in $G_t$). We claim that in $G_t$ removing $e_1, \ldots, e_k$ separates $p_t(v)$ from $p_t(w)$. Otherwise, there is a $p_t(v)-p_t(w)$ path $P$ in $G_t - \Set{e_1,\ldots,e_k}$. But then, due to the monotonicity of $p_t$, the subspace $p_t^{-1}(P)$ is a connected subset of $X - \Set{e_1,\ldots,e_k}$ containing both $v$ and $w$, a contradiction.

Since every bonding map, $\pi_n$ from $G_n$ to $G_{n-1}$ is nice, it follows from Lemma~\ref{ec_nice} that $(k_{G_n} (p_n(v),p_n(w)))_{n \ge s}$ is indeed decreasing. So it must stabilize at some $t$, with value $k_X(v,w)$. It follows that in $X$ there are open edges $E_1, \ldots, E_k$, where $k=k_X(v,w)$, separating $v$ from $w$, such that these same edges exist in $G_t$.
From the argument above we see that -- as claimed -- $t$ is minimal for which there is a set $\mathcal{E}$ of edges in $X$ of size $k_X(v,w)$ separating $v$ and $w$ such that all members of $\mathcal{E}$ exist in $G_t$.
\end{proof}


\section{The Graph-Like Examples}\label{sec_4}

In this section we construct families of examples which demonstrate that -- with the sole exception of the characterization of $2$-cc graph-like continua given in Proposition~\ref{gl-2cc} -- none of our positive results of Section~\ref{sec_2} for $n$-ac and $n$-cc Freudenthal compactifications of locally finite graphs extend to arbitrary graph-like continua. Below we write $K_m$ for the complete graph on $m$  vertices.


\subsection{A procedure for constructing graph-like continua.} \













Every graph-like continuum, $X$ say, can be represented as an inverse limit, $\varprojlim G_k$, of connected graphs, with nice bonding maps. The $k$th bonding map, $\pi_k$, determines how to transition from $G_{k+1}$ to $G_k$. 

For the purposes of \emph{constructing} a graph-like continuum, however, it is more convenient to have a rule for building $G_{k+1}$ from $G_k$, and then specifying the bonding map. For our present purposes the following method is simple but effective.

The input data for the construction process are: (1)~the first graph, $G_1$, and (2)~rules, one for each $n$, specifying how to replace a vertex, $v$, of degree $n$ in a graph by a connected subgraph, $G_v$. Then to construct the inverse sequence, recursively apply the rules to the vertices of $G_k$ to get $G_{k+1}$, and define the bonding map $\pi_k$  to be the map which collapses each connected subgraph, $G_v$, in $G_{k+1}$ to $v$ in $G_k$. Clearly this map is nice.

By convention, if no rule is specified for vertices of degree $n$, then the rule is to leave the vertex alone. A typical rule for vertices of degree four is depicted below. Here each vertex of degree four is to be replaced with the complete graph on four vertices, and the four original edges are connected to one new vertex of the complete graph each. The bonding map collapses the new complete graph to the single old vertex.
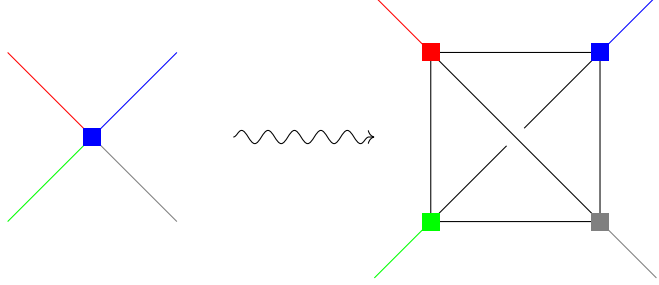
\begin{center}
\begin{tikzpicture}[scale=3.75]

\draw[->,decorate,decoration={snake}] (-0.5,-0.5) -- (0.0,-0.5); 


\draw[red] (-1.3,-0.2) -- (-1.0,-0.5);
\draw[blue] (-0.7,-0.2) -- (-1.0,-0.5);
\draw[green] (-1.3,-0.8) -- (-1.0,-0.5);
\draw[gray] (-0.7,-0.8) -- (-1.0,-0.5);

\node[fill=blue] at  (-1.0,-0.5) {};


\draw (0.2,-0.8) -- ++(0,0.6) -- ++(0.6,0) -- ++(0,-0.6) -- ++(-0.6,0) --++(0.6,0.6);

\draw[green] (0.2,-0.8) -- (-0,-1);
\draw[red] (0.2,-0.2) -- (-0,0);

\draw[gray] (0.8,-0.8) -- (1,-1);
\draw[blue] (0.8,-0.2) -- (1,0);

\node[fill=white, circle] at (1/2,-1/2) {};
\draw (0.2,-0.2) -- ++(0.6,-0.6);

\node[fill=red] at  (0.2,-0.2) {};
\node[fill=green] at  (0.2,-0.8) {};
\node[fill=blue] at  (0.8,-0.2) {};
\node[fill=gray] at  (0.8,-0.8) {};

\end{tikzpicture}
\end{center}

\subsection{Non-trivial $\omega$-ac and $\omega$-cc graph-like continua} \ 

\begin{exam}\label{ex_omcc}
There is a graph-like continuum which is $\omega$-cc but is not a graph (in particular, not the circle). 
\end{exam}


\begin{wrapfigure}[27]{r}{3.75cm}

\begin{tikzpicture}[scale=2.6]

 
\node at (0.1,0.9) {$G_1$};
\node at (0.1,0.1) {$G_2$};
\node at (0.1,-1.4) {$G_3$};



\draw[red] (0.25,0.6) circle [radius=0.25];
\draw[red] (0.75,0.6) circle [radius=0.25];

\node[fill=blue] at  (0.5,0.6) {};





\draw (0,-1) -- ++(0,1) -- ++(1,0) -- ++(0,-1) -- ++(-1,0) --++(1,1);

\draw[red] (0,-1) .. controls (-0.2,-1.4) and (-0.2,0.4) .. (0,0);

\draw[red] (1,-1) .. controls (1.2,-1.4) and (1.2,0.4) .. (1,0);

\node[fill=white, circle] at (1/2,-1/2) {};
\draw (0,0) -- ++(1,-1);

\node[fill=blue] at  (0,0) {};
\node[fill=blue] at  (0,-1) {};
\node[fill=blue] at  (1,0) {};
\node[fill=blue] at  (1,-1) {};









\draw (0,-2.5) -- ++(0,1) -- ++(1,0) -- ++(0,-1) -- ++(-1,0) --++(1,1);

\draw[red] (0,-2.5) .. controls (-0.2,-2.9) and (-0.2,-1.1) .. (0,-1.5);

\draw[red] (1,-2.5) .. controls (1.2,-2.9) and (1.2,-1.1) .. (1,-1.5);

%



%

\node[fill=white, circle] at (1/2,-2) {};
\draw (0,-1.5) -- ++(1,-1);


%

\node[fill=white,circle] at (0.2,-2.3) {};
\node[fill=white,circle] at (0.2,-1.7) {};
\node[fill=white,circle] at (0.8,-2.3) {};

\draw (0.4,-2.5) --++(0,0.4) --++(-0.4,0) -- ++(0.4,-0.4);

\draw (0.4,-1.5) --++(0,-0.4) --++(-0.4,0) -- ++(0.4,0.4);

\draw (0.6,-2.5) --++(0,0.4) --++(0.4,0) -- ++(-0.4,-0.4);
\node[fill=white,circle] at (0.8,-1.7) {};
\draw (0.6,-1.5) --++(0,-0.4) --++(0.4,0) -- ++(-0.4,0.4);







\node[fill=blue] at  (0,-2.5) {};
\node[fill=blue] at  (0,-1.5) {};
\node[fill=blue] at  (1,-2.5) {};
\node[fill=blue] at  (1,-1.5) {};
%

%
\node[fill=blue] at  (0.4,-2.5) {};
\node[fill=blue] at  (0.4,-2.1) {};
\node[fill=blue] at  (0,-2.1) {};
%

%
\node[fill=blue] at  (0.4,-1.5) {};
\node[fill=blue] at  (0.4,-1.9) {};
\node[fill=blue] at  (0,-1.9) {};
%

%
\node[fill=blue] at  (0.6,-2.5) {};
\node[fill=blue] at  (0.6,-2.1) {};
\node[fill=blue] at  (1,-2.1) {};
%

%
\node[fill=blue] at  (0.6,-1.5) {};
\node[fill=blue] at  (0.6,-1.9) {};
\node[fill=blue] at  (1,-1.9) {};
%

\end{tikzpicture}
\end{wrapfigure}
\ \\ 
\vspace*{-28pt}

\begin{proof}[Construction]
For each $k$ we define recursively, $4$-regular (multi) graphs $G_k$, following the procedure outlined above.  The graph-like continuum $X=\varprojlim G_k$ will be $\omega$-cc, but not a graph.

Let $G_1$ be any $4$-regular connected multi-graph, for example the figure-eight graph (one vertex, two loops). The rules for constructing $G_{k+1}$ from $G_k$ are always the same: uncontract every vertex of $G_k$ to a complete graph on four vertices, $K_4$, in the natural manner (as above). This will have the effect that $G_{k+1}$ will still be $4$-regular, and so the recursion can be continued. The first three steps of the algorithm are depicted right.

It is obvious that $X$ is not a graph. To see that $X=\varprojlim G_k$ is $\omega$-cc, let $n \in \N$ be arbitrary, and note that by Lemma~\ref{lem_wlogpointsonedges3} it suffices to consider points $x_1, \ldots , x_n$ lying on (different) edges of $X$. Find $k \in \N$ sufficiently large such that $x_1, \ldots , x_n$ lie on different edges of $G_k$. Since $G_k$ is $4$-regular, it has an Eulerian cycle $\alpha$. Since in $G_{k+1}$, every vertex of $G_k$ is expanded into a $K_4$, is is easy to see that the cycle $\alpha$ lifts to a simple closed curve $\alpha'$ of $G_{k+1}$, containing all vertices $x_1, \ldots , x_n$. By Lemma~\ref{lem_lifting}, $\alpha'$ lifts to a simple closed curve $\alpha''$ of $X$ containing all vertices $x_1, \ldots , x_n$, and so the proof is complete.
\end{proof}

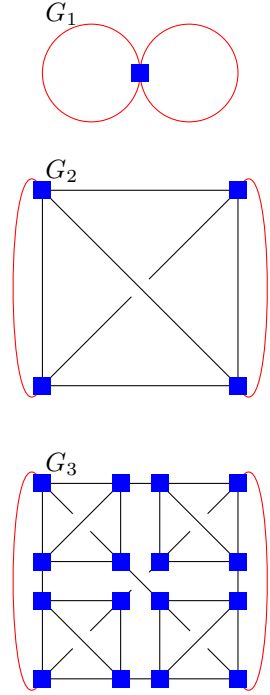
\emph{Note:} for the above construction to produce an $\omega$-cc graph-like continuum it suffices that (1) every $G_k$ is Eulerian and (2) 
each vertex $v$ in some $G_k$ is uncontracted to $G_v$ in $G_{k+1}$ so that every edge in $G_k$ incident to $v$ is incident to distinct vertices in $G_v$, and those vertices are contained in a complete subgraph of $G_v$. (That each $G_k$ is $d_k$-regular, and $(d_k)_k$ is constant, simplifies defining the expansion rules, but neither constraint is necessary.)

\begin{exam}\label{ex_omac} There is a graph-like continuum $X$, not a graph, which is $\omega$-ac but not $2$-cc.
\end{exam}

\begin{proof}[Construction]
Indeed, such examples can easily be constructed by considering a figure-eight-curve, a dumbbell, or a lollypop-curve, and replacing one of the circles in these graphs by a copy of the $\omega$-cc graph-like continuum from the previous example. 
\end{proof}

\begin{theorem}
\label{thm_manydifferentomegacc} \ 

(a) There are $2^{\aleph_0}$ many pairwise non-homeomorphic $\omega$-cc graph-like continua. 

(b) There are $2^{\aleph_0}$ many pairwise non-homeomorphic $\omega$-ac, but not $2$-cc, graph-like continua. 
\end{theorem}

\begin{proof} From Example~\ref{ex_omac} it is clear that (b) follows from (a).

Let $G_1$ be the graph on a single vertex with a single loop. Take any function $f \in \N^\N$ which is strictly increasing, and for every $n$ we have $f(n)$ divisible by $2$. 
Define $X_f =\varprojlim G^f_k$ where the graphs $G^f_k$ are given recursively by:
\begin{itemize}
\item $G^f_1=G_1$, and
\item  $G^f_{k+1}$ is obtained from $G^f_k$ by uncontracting every vertex $v$ of $G^f_k$ to a $\tilde{K}_{f(k)} \supseteq K_{f(k)}$, where the edges incident with $v$ are incident with distinct vertices of $\tilde{K}_{f(k)}$ and the remaining vertices of $K_{f(k)}$ get paired up, and get an additional parallel edge between each pair as to satisfy the even degree condition.
\end{itemize}
Note that, inductively, each $G^f_{k+1}$ is a connected, $f(k)$-regular graph (hence, as $f(k)$ is even, Eulerian), and this combined with the fact that $f$ is strictly increasing  and has even values ensures that $G^f_{k+1}$ is well-defined from $G^f_k$.

The graphs $G^f_k$ satisfy properties (1) and (2) noted after Example~\ref{ex_omcc}, from which it follows that the graph-like continuum $X_f$ is $\omega$-cc.

{\bf Claim 1:} \emph{If $v$ and $w$ are distinct vertices of $G^f_{k+1}$ which are projected to the same vertex $x$ of $G^f_k$, then $f(k)-1 \leq k_{G^f_{k+1}} (v ,w) \leq f(k)$.} 

By $f(k)$-regularity of $G^f_{k+1}$, the edge-connectivity is at most $f(k)$. The first inequality holds since the complete graph $K_{f(k)}$ has edge-connectivity $f(k)-1$. 

{\bf Claim 2:} \emph{If $v$ and $w$ are vertices of $G^f_{k+1}$ such that their projections $v'=\pi_{k}(v)$ and $w'=\pi_{k}(w)$ are distinct in $G^f_{k}$, then $k_{G^f_{k+1}} (v,w) = k_{G^f_{k}} (v',w')$.}

By Lemma~\ref{ec_nice}, it suffices to show $k_{G^f_{k+1}} (v,w) \ge k_{G^f_{k}} (v',w')=k$. But this inequality follows from Menger's theorem, since there is a collection  of $k$-many edge-disjoint $v'-w'$-paths in $G^f_k$ which lift, by the fact that we uncontracted vertices to complete graphs and by property (2), to a collection  of $k$-many edge-disjoint $v-w$-paths in $G^f_{k+1}$, establishing the claim.

Next, define $\script{C}_f=\set{k_{X_f}(v,w)}:{v\neq w \in V(X_f)}$, the spectrum of  all edge-connectivities between pairs of distinct vertices of $X_f$. From Claims~1 and~2, along with Lemma~\ref{edge_c} we deduce:

{\bf Claim 3:}\emph{
\begin{enumerate}
\item $\script{C}_{f} \subseteq  \set{f(n)-1}:{n \in \N} \cup \set{f(n)}:{n \in \N}$, and
\item for each $n \in \N$ we have $\Set{f(n)-1,f(n)} \cap \script{C}_{f} \neq \emptyset$.
\end{enumerate}}

\smallskip

\noindent Now define $\script{F} = \set{f\in \N^\N}:{f \text{ is strictly increasing and }\forall n \,  f(n)\text{ is even}}$. Then $|\script{F}| = 2^{\aleph_0}$. For each $f \in \script{F}$ we know   $X_f =\varprojlim G^f_k$ is an $\omega$-cc graph-like continuum, and we now show these are pairwise non-homeomorphic.

{\bf Claim 4:} \emph{For distinct $f \neq g \in \script{F}$, the graph-like continua $X_f$ and $X_g$ are non-homeomorphic.}

To see this, let $k \in \N$ be minimal such that $f(k) \neq g(k)$, and without loss of generality assume that $f(k) < g(k)$. Note that $k \ge 2$ (since $G^f_1=G^g_1$). As $f,g$ are strictly increasing and have even values, we have $f(k-1)=g(k-1) < f(k)-1 < f(k) < g(k)-1<g(k)$.  
Hence, from Claim~3, one of $f(k)-1$ and $f(k)$ is in $\script{C}_f$ but neither is in $\script{C}_g$, so $\script{C}_f \setminus \script{C}_g \neq \emptyset$, and so we deduce $X_f \not\cong X_g$ by Lemma~\ref{lem_homeomorphismequalsisomorphism}.
\end{proof}

\subsection{Graph-like continua which are $n$- but not $(n+1)$- ac or cc}

In this section, we construct interesting graph-like continua which are $n$-ac but not $(n+1)$-ac, and others which are $n$-cc but not $(n+1)$-cc. For these, we present two fundamentally different constructions. 

The first construction uses knowledge about certain closed or open Eulerian paths in finite minors of the graph-like space. In some sense, this first construction is all about controlling the edge-cuts in the space. 
The second construction starts with several copies of a graph-like space, in which we have a lot of control over which arcs we may use to pick up our favorite edge set. We then glue together these copies by identifying some finite set of vertices. In some sense, this second construction is all about controlling the vertex-cuts in the space.

\subsubsection{Technique 1: Using open and closed Eulerian paths in finite graphs}

For our next examples, we need the following auxiliary result. Recall that a \emph{matching} in a graph is a collection of pairwise non-adjacent edges.

\begin{lemma}
\label{lem_largecomplete}
For every $n \ge 2$, the complete graph on $N \geq 4n+4$ vertices has the property that given (i) any matching $M$ in $K_{N}$,
  (ii) any edges $e_1, \ldots, e_k$ of $K_N - M$ with $k \leq n$, and
 (iii) any two vertices $v, w$ in $K_N$,
there is a non-edge-repeating trail from $v$ to $w$ in $K_{N} - M$ containing the selected edges.
\end{lemma} 

\begin{proof}
To see the claim, note that after removing the matching $M$, every vertex has degree at least $N-2$ in the subgraph $H_0 = K_{N} - M$, and so any two vertices have at least $N-4$ common neighbours in $H_0$. Write $e_i = x_iy_i$. Since $v$ and $x_1$ have a common neighbour, there is a path $P_1$ from $v$ to $y_1$ with $e_1 \in E(P_1)$. Next, consider $H_1 = H_0 - E(P_1)$ and note that every vertex in $H_1$ has degree at least $N-4$, and so any two vertices have at least $N-8 \geq 4n-4 > 0$ common neighbours in $H_1$. If $e_2$ isn't yet covered by $P_1$, find a path $P_2$ in $H_1$ from $y_1$ to $y_2$ containing the edge $e_2$. If we continue in this manner, then in $H_k = H_0 \setminus \bigcup_{i \leq k} E(P_i)$, every vertex has degree at least $N-2 -2k \geq N/2$.
Hence, any two vertices in $H_k$ are either are connected by an edge, or have a common neighbour. Thus, there is a path $P_{k+1}$ in $H_k$ from $y_k$ to $v$. It is clear that $\bigcup_{i \leq k+1} P_i$ is the desired edge trail.
\end{proof}

\begin{exam}
For each $n \ge 2$ there is a graph-like continuum which is $n$-ac but not $(n+1)$-ac.
\end{exam}

\begin{proof}[Construction]
Fix $n \ge 2$. We define a sequence of graphs, $G_k^n$, by giving the first, $G_1^n$, then $G_2^n$, and a rule defining $G_{k+1}^n$ from $G_k^n$, for $k \ge 2$. This naturally gives an inverse limit $X_n=X=\varprojlim G_k^n$ which is graph-like. 

{\bf Case 1:} \emph{$n=2m+1$ is odd where $m \ge 1$.} The graph $G_1^n$ has four vertices, $v_1, w_1, w_2$ and $v_2$. There is an edge connecting $v_i$ to $w_i$ for $i=1,2$; and $2m$ edges connecting $w_1$ and $w_2$. Thus $G_1^n$ has $n+1$ edges, two vertices of degree $1$ and two of degree $n$. It is easy to check that $G_1^n$ is $n$-oE. But $G_1^n$ is not $(n+1)$-oE, and so by Proposition~\ref{necc}(b) $X$ is not $(n+1)$-ac. Next, let $N=N(n)$ be large enough as to satisfy Lemma~\ref{lem_largecomplete}. To define $G_{2}^n$ from $G_1^n$ leave the two vertices of degree $1$ alone, and uncontract the two vertices of degree $n$ to a $K_N$, such that all vertices of $G_2^1$ are either of degree $1$, $N-1$, or $N$. To define $G_{k+1}^n$ from $G_k^n$ leave the (two) vertices of degree $1$ alone, and replace all vertices of degree $N-1$ or $N$ by a complete graph on $N$ new vertices. Since all vertices of $G_k^1$ are either of degree $1$, $N-1$ or $N$, inductively, the same is true for $G_k^n$, and then $G_{k+1}^n$. Hence the definition is complete.

We now show by induction on $k$ that for all $k$ the graph $G_k^n$ is $n$-oE. Then the proof that $X$ is $n$-ac then follows as in the previous examples.
Fix $k \ge 2$. Let $\pi=\pi_k \colon G_{k+1}^n \to G_k^n$ be the bonding map. Take any subset $S$ of $G_{k+1}^n$ containing no more than $n$ points. Then, inductively, in $G_k^n$ there is an edge-disjoint trail containing $\pi(S)$. The edges in this trail pull back to an edge-disjoint sequence of (directed) edges in $G_{k+1}^n$ so that successive edges have end and start points (respectively) mapping to the same vertex in $G_k^n$.
We explain how to add edges in fibers of vertices of $G_k^n$ so as to form an edge-disjoint trail in $G_{k+1}^n$ containing the points of $S$.

It suffices to consider one vertex $v$ of $G_k^n$, and add edges in $\pi^{-1} \{v\}$ so as to connect together successive edges in the edge-disjoint sequence while preserving edge-disjointness and ensuring that all points in $S$ which happen to lie in $\pi^{-1} \{v\}$ are contained in the resulting trail. 
If $\pi^{-1} \{v\}$ is just one point then there is nothing to do. Otherwise $\pi^{-1} \{v\}$ is a complete graph on $N$ vertices. If no edges in the edge-disjoint sequence meet $\pi^{-1} \{v\}$ there is nothing to do.
List all successive pairs entering and exiting $\pi^{-1} \{v\}$ as $e_1^0, e_2^0$, $e_1^1,e_2^1,  \ldots , e_1^p, e_2^p$, where $p \ge 0$. Let $f_1, \ldots , f_q$ be the edges in $\pi^{-1} \{v\}$ containing points of $S$. Note $q \leq n$.

For $i=1, \ldots , p-1$ add the edge in $\pi^{-1} v$ connecting the end of $e_1^i$ to the start of $e_2^i$. By construction, this edge set is a matching $M$. If at this point, some of the edges $f_i$ are yet uncovered, we may add, by Lemma~\ref{lem_largecomplete}, a trail from the end of $e_1^p$ to the start of $e_2^p$ disjoint from $M$ in $\pi^{-1} v$ containing all uncovered edges of $f_1, \ldots , f_q$. Otherwise, simply add the edge in $\pi^{-1} v$ connecting the end of $e_1^p$ to the start of $e_2^p$. Now we are done.

{\bf Case 2:} \emph{$n=2m$ is even where $m \ge 1$.} The graph $G_1^n$ has four vertices, $v_1, w_1, w_2, v_2$. There are $n-1$ edges connecting $w_1$ and $w_2$, and one edge from each of $v_1$ and $v_2$ to $w_1$. Then $G_1^n$ has $n+1$ edges, two vertices of degree $1$, one of degree $n+1$ and one of degree $n-1$. It is easy to check that $G_1^n$ is $n$-oE but not $(n+1)$-oE.   

Let $N=N(n+1)$ be large enough as to satisfy Lemma~\ref{lem_largecomplete} for $n+1$. Define $G_2^n$ by replacing the single vertex of degree $n+1$ with $N$ new vertices connected by a complete graph, but leaving the other vertices alone.
To define $G_{k+1}^n$ from $G_k^n$ leave the two vertices of degree $1$ alone, leave the vertex of degree $n-1$ alone, and replace all vertices of degree $N$ or $N-1$ with $N$ new vertices and a complete graph connecting them. Now the argument that $X=\varprojlim G_k^n$ is as required is very similar to that given above in Case~1.
\end{proof}

\begin{exam}
For each even $n$ there is a graph-like continuum which is $n$-cc but not $(n+1)$-cc.
\end{exam}

\begin{proof}[Construction]
The argument is similar to that given above for graph-like continua which are $n$-ac but not $(n+1)$-ac. So we give a sketch only, highlighting differences.

Fix even $n$. Let $G_1^n$ be the (multi-)graph with two vertices and $n+1$ parallel edges connecting them. Note that the vertices have degree $n+1$, and it is easy to check $G_1^n$ is $n$-E (given any $n$ points there is a closed edge-disjoint trail containing them). Pick $N=N(n+1)$ be large enough as to satisfy Lemma~\ref{lem_largecomplete} for $n+1$. Recursively define $G_{k+1}^n$ from $G_k^n$ by uncontracting each vertex to a $K_N$. By induction one can check that every $G_k^n$ is $n$-E.

Define $X=\varprojlim G_k^n$. Then $X$ is a graph-like continuum, and arguing as before it can be verified to be $n$-cc. But picking a point from the interior of each edge easily shows $G_1^n$ is not $(n+1)$-E. Hence, by Proposition~\ref{necc}, $X$ is not $(n+1)$-cc.
\end{proof}



Our strategy from above is bound to fail when trying to build an example for a graph-like continuum which is $n$-cc but not $(n+1)$-cc for odd $n$. 
Indeed, given odd $n$ we would need graphs which are $n$-E but not $(n+1)$-E, however the second author and Knappe have shown that this is impossible -- any graph which is $n$-E, where $n$ is odd, is automatically $(n+1)$-E, see \cite{KP}.
Hence, a fundamentally different approach is required to construct, for odd $n$, graph-like continua which are $n$-cc but not $(n+1)$-cc. This is the purpose of our next and final section.

\subsubsection{Technique 2: Using small vertex cuts in graph-like spaces}

Recall that in an $n+1$-ac graph-like continuum, deleting $n-1$ vertices creates at most $n$ distinct connected components, \cite[Lemma 2.3.3]{GMP}

A similar result holds for $(n+1)$-cc graphs: Recall that a connected graph, or a graph-like continuum $G$ is called \emph{$k$-tough}, if for any finite, non-empty set of vertices $S$, the number of components of $G - S$ is at most $|S|/k$. Adapting this notion slightly, let us say that a graph-like continuum $G$ is \emph{$(k,n)$-tough} if for any set of vertices $S$ with $1 \leq |S| \leq n$, the number of components of $G - S$ is at most $|S|/k$.

The standard notion of toughness plays a well-known role in the theory of Hamilton cycles, as a necessary condition for a finite graph to be Hamiltonian is that it is $1$-tough. The straightforward adaptation of this result to our use case gives the following observation.

\begin{lemma}
\label{lem_Max2}
Every $(n+1)$-cc graph-like continuum is $(1,n)$-tough.
\end{lemma}

\begin{proof}
Suppose $X$ is an $n$-cc graph-like continuum and, for a contradiction, $S \subset V(X)$ is a finite vertex set with $1 \leq |S|=s \leq n$ whose removal leaves  strictly more than $s$ components. Pick $s+1$ edges in different components of $X - S$. As $s+1 \leq n+1$, by assumption, there is a simple closed curve $\alpha$ in $X$ picking up the edges. But then $\alpha \setminus S$ consists of at most $s$ components. Hence, there are two edges in the same component of $\alpha \setminus S$, contradicting the fact that they lie in different components of $X - S$.
\end{proof}

As our building blocks, we will use the following class of graphs.

\begin{exam}
\label{ex_Max1}
For each $n \geq 2$ there is a graph-like continuum $X$ containing vertices $v_1,v_2,\ldots, v_n$ such that 
(i) whenever an edge set $F \subset E(X)$ with $|F| \leq n$ is chosen, and
(ii) any two vertices $v_i \neq v_j$ from our list are chosen, 
there is an $v_i-v_j$ arc $\alpha$ in $X$ containing $F$ but not $v_k$ for all $k \ne i,j$.
\end{exam}

\begin{proof}
Let $n \in \N$ be fixed and consider $N=N(n)$ from Lemma~\ref{lem_largecomplete}. We will construct $X$ as an inverse limit of finite graphs $G_n$ where we start with $G_1 = K_N$, and uncontract in each step every vertex $v$ of $G_k$ to a new $K_N$. It follows recursively that every vertex of $G_k$ has degree $N$ or $N-1$. 

Let $p_k \colon X \to G_k$ denote the quotient map. Choose $v_1,\ldots,v_n \in V(X)$ subject to the condition that the degree of $p_k(v_i)$ equals $N-1$ for each $k \in \N$.
Now pick any edge set $F$ with $|F| \leq n$. We will demonstrate that there is an $v_1-v_2$ arc $\alpha$ in $X$ with $F \subset \alpha$ and $v_i \notin \alpha$ for all $i \geq 3$.

By Lemma~\ref{lem_largecomplete}, there is a $p_1(v_1) - p_1(v_2)$-trail $T_1$ in $G_1$ containing $F \cap E(G_1)$. Recursively, using again Lemma~\ref{lem_largecomplete}, extend this to an $p_k(v_1) - p_k(v_2)$-trail $T_k$ in $G_k$ containing $F \cap E(G_k)$ until $F \cap E(G_k) = F$. Next, using the fact that $p_{k+1}(v_i)$ equals $N-1$, extend $T_k$ to an $p_{k+1}(v_1) - p_{k+1}(v_2)$-path $T_{k+1}$ in $G_{k+1}$ missing all $p_{k+1}(v_i)$ for all $i \geq 3$. Extending this path $T_{k+1}$ recursively, it is clear that we end up with the desired $v_1-v_2$-arc.
\end{proof}

\begin{exam}\label{nodd2cc}
For each $n \ge 2$ there is a graph-like continuum which is $n$-cc but not $(n+1)$-cc.
\end{exam}

\begin{proof}[Construction]
Let $X$ be the space from Example~\ref{ex_Max1} with special points $v_1, \ldots, v_n$. Now take $n+1$ many disjoint copies $X^{(1)}, \ldots, X^{(n+1)}$ of the space $X$ with the special points denoted by $v^{(i)}_1, \ldots, v^{(i)}_n \in V(X^{(i)})$.

We claim the graph-like continuum
\[Z = \p{X^{(1)} \oplus \cdots \oplus X^{(n+1)}} /_{\sim} \; \text{ where } v^{(1)}_k \sim v^{(2)}_k \sim \cdots \sim v^{(n+1)}_k \; \text{ for each } k,\]
is $n$-cc but not $(n+1)$-cc. Let us write $[v_k] \in Z$ for the vertex corresponding to the equivalence class of $v^{(1)}_k$. Then it is clear from the construction that deleting $S = \Set{[v_1], \ldots, [v_n]}$
from $Z$ leaves $n+1$ many components. Therefore, $Z$ is not $(1,n)$-tough, and hence cannot be $(n+1)$-cc by Lemma~\ref{lem_Max2}.

To see that $Z$ is $n$-cc, consider any collection $F = \Set{e_1,e_2,\ldots, e_n}$ of $n$ edges of $Z$ (which is sufficient because of Lemma~\ref{lem_wlogpointsonedges3}). We may assume that the edges are contained in the first $i$ spaces $X^{(1)} \cup \cdots \cup X^{(i)}$ where $i \leq n$. By the properties guaranteed by example~\ref{ex_Max1}, we can find $v^{(j)}_j-v^{(j)}_{j+1}$ arcs $\alpha^{(j)}$ (where $i+1 \equiv 1$) in $X^{(j)}$ missing all other special vertices and containing $F \cap E(X^{(j)})$. 
It is then clear that $\alpha := \bigcup_{j \leq i} \alpha^{(j)} \subset Z$ is the desired simple closed curve in $Z$ containing $F$ (as each $\alpha^{j}$ and $\alpha^{j+1}$ end and start at the same vertex $[v_{j+1}] \in Z$ respectively, and $\alpha^{j}$ and $\alpha^{\ell}$ are disjoint for $|(j - \ell \pmod{n})| \geq 2$).
\end{proof}



\begin{theorem} \label{thm_manydifferentother} For every $n \ge 2$:

(a)${}_n$ there are $2^{\aleph_{0}}$ many non-homeomorphic graph-like continua which are $n$-ac
but not $(n+1)$-ac, and

(b)${}_n$ there are $2^{\aleph_{0}}$ many non-homeomorphic graph-like continua which are $n$-cc
but not $(n+1)$-cc.
\end{theorem}

\begin{proof}
This follows by the same method as we derived Theorem~\ref{thm_manydifferentomegacc} (a) from Example~\ref{ex_omcc} with some small adjustments that we show here.

Fix $n$. 
Both techniques to construct `$n$-ac not $(n+1)$-ac' and `$n$-cc not $(n+1)$-cc' graph-like continua used Lemma~\ref{lem_largecomplete} to replace vertices by a big enough $K_N$ where $N$ depended on $n$. 

As in Theorem~\ref{thm_manydifferentomegacc}, let $\script{F} = \{f\in \N^\N: f$ is strictly increasing, $\forall n \, f(n)$ is divisble by $4$,  and $f(1) \geq N\}$. Then $|\script{F}| = 2^{\aleph_0}$. To define the sequence of graphs, $G^f_k$, at step $k+1$ uncontract vertices in the $k$th step into a $K_{f(k)}$.  

Then $X_f = \varprojlim G_k^f$ is a graph-like continuum with the requisite combination of strong connection properties (`$n$-ac not $(n+1)$-ac' or `$n$-cc not $(n+1)$-cc').
And, as in the proof of Theorem~\ref{thm_manydifferentomegacc}, for distinct $f$ and $g$ from $\script{F}$ the spaces $X_f$ and $X_g$ have different edge-connection spectra, and so are non-homeomorphic.

In all cases \emph{except} for the construction of an $n$-cc not $(n+1)$-cc graph-like continuum where $n$ is odd, these $X_f$ are as needed.
But for `odd $n$, $n$-cc not $(n+1)$-cc' we require an extra step as in Example~\ref{nodd2cc}. There, for each $f$ in $\mathcal{F}$, the final example, $Z_f$ is obtained  by gluing $n+1$-many copies of $X_f$.
So it remains to show that for distinct $f$ and $g$ from $\mathcal{F}$ the spaces $Z_f$ and $Z_g$ are non-homeomorphic.

However, it follows from Proposition~\ref{kconn} that each $X_f$ has vertex connectivity $\ge f(1)\geq N > n$. So when gluing $(n+1)$ copies together over an $n$-point set to form $Z_f$, this set is the unique vertex separator of size $n$ in $Z_f$. Since this separator must be preserved by any homeomorphism we see that indeed distinct $f$ and $g$ yield topologically distinct $Z_f$ and $Z_g$.
\end{proof}

\iftrue

\appendix
\section{Computer verification for example~\ref{onepointExs}}

We record here why $\alpha C$ of Example~\ref{onepointExs} is $6$-ac. For this, let $V(C) = \Set{0,1} \times \Z$. Two vertices $(m,n)$ and $(m',n')$ are adjacent if and only if $|m-m'|+ |n-n'|=1$, without the edge $\Set{(0,0),(0,1)}$. Moreover, let us subdivide the edges $\Set{(0,-2),(0,-3)}$ and $\Set{(0,3),(0,4)}$  by vertices $a$ and $b$, and add new edges from $a$ to $(0,0)$ and $(0,1)$ to $b$.
Let us write $e = \Set{(1,0),(1,1)}$ for the unique bridge of $C$, and $C_{+} = C[\Set{0,1} \times \N]$ and $C_{-} = C[\Set{0,1} \times -\N_0]$.

Pick any six points $x_1,\ldots,x_6$  from $\alpha C$. We may suppose they lie on distinct edges. Naturally, some of the points will be contained in $C_+$, some in $C_{-}$, and additionally, we may assume that at most one point lies on $e$. 

\textbf{Case A:} \emph{Either $C_+ \cup e$ or $C_{-} \cup e$ contains all six points. } By symmetry, it suffices to deal with the case where $C_+ \cup e$ contains all six points. Find $n\geq 4$ large enough such that $x_1,\ldots,x_6 \in e \cup C_+[\Set{0,1} \times \Set{0,1,\ldots,n}]$.

Now let $G_1:=e \cup C_+[\Set{0,1} \times \Set{0,1,\ldots,n}]$, take a further disjoint copy $G_2$ of $G_1$, and consider the auxiliary graph $G= G_1 \sqcup G_2 /{\sim}$ where we identify the respective leaves (endpoints of degree $1$) of the edge $ e$, and add one new edge $f$ between the copies of $(0,n)$, and one further new edge $g$ between the copies of $(1,n)$.

It follows from \cite[Theorem~3.4.1]{GMP} that $G$ is $6$-ac, and so there is an arc $\alpha$ in $G$ containing $x_1,\ldots,x_6$. Without loss of generality, $\alpha$ starts and ends in points $x_i$, and so in particular it starts and ends outside of $G_2$. Moreover, note that $\partial G_2 = \Set{e,f,g}$ is a 3-edge cut, and so if $\alpha$ contains points from $G_2$ then $\alpha$ will cross this cut in precisely two edges, and so $\beta = \alpha \cap G_2$ will be a subarc of $\alpha$. But then it is clear that by replacing $\beta$ with a suitable arc in $\alpha C \setminus G_1$, we may lift $\alpha$ to an arc in $\alpha C$ containing all six points. And of course, if $\alpha \cap G_2 = \emptyset$, then $\alpha \subset G_1 \subset \alpha C$ is already an arc witnessing $6$-ac.

\textbf{Case B:} \emph{$C_+ \cup e$ contains $5$ points and $C_{-}$ contains one. } This case is very similar to the previous case. Indeed, since there are at most $5$ of our points contained in $e \cup C_+[\Set{0,1} \times \Set{0,1,\ldots,n}]$, we may place an additional point $y$ on $e$, and apply the previous construction to see that there is an arc $\alpha_+$ in $C_+$ containing all points $\Set{x_1,\ldots,x_5,y}$. By choice of $y$, this arc $\alpha$ is forced to use the edge $e$. Now it is clear that we may lift this to an arc $\beta$ in $\alpha C$ by replacing $\alpha \restriction e$ with a suitably $(1,1)-\infty$ path in $C_{-}$ picking up the remaining point $x_6$ (using $3$-sac).

\textbf{Case C:} \emph{$C_+$ contains $3$ points and $C_{-}$ contains $3$. } This case is straightforward: find $n \in \N$ large enough such that $x_1,\ldots,x_6 \in e \cup C[\Set{0,1} \times \Set{-n, -n+1,\ldots, n-1,n}]$. Then $C_{-}\cup e$ and $C_{+}\cup e$ are $4$-ac by \cite[Theorem~3.2.3]{GMP}, so by placing an additional point on $e$ in both sides, we obtain arcs $\alpha_-$ and $\alpha_+$ in $C_{-}\cup e$ and $C_{+}\cup e$ containing all points $x_i$ and both starting with the edge $e$. It is then clear that $\alpha_- \cup \alpha_+$ is the desired arc.

\textbf{Case D:} \emph{$C_+ \cup e$ contains $4$ points and $C_{-}$ contains $2$. } Note that by the previous argument, may assume we are in the situation where \emph{$C_+$ contains $4$ points and $C_{-}$ contains $2$, and no point on $e$. }

Clearly the two points in $C_-$ are contained in an arc in $\alpha C_-$ that ends at the point at infinity in $\alpha G$, and in another arc in $C_- \cup  e$ that ends at the end, $(0,1)$ of $e$. So it suffices to show that any $4$ points in $C_+$ are either contained in an arc in $\alpha C_+$ that ends at the point at infinity or at $(0,1)$.

So fix $4$ points $x_1, \ldots , x_4$ on distinct edges of $C_+$. If all four points lie on (horizontal) rungs then a simple zig-zag arc contains them, and can be extended to the point at infinity. 
So assume at most $3$ points lie on  rungs.

Consider those points from $x_1, \ldots , x_4$ (if any) which are in $C_+[\{0,1\}\times [4,\infty)]$. 
By deleting some (horizontal) rungs and merging successive (vertical) edges we can assume they are in $C_+[\{0,1\}\times [4,5,6,7]]$, and so all of $x_1, \ldots , x_4$ are in $C_+[\{0,1\} \times [0,7]]$.

Let $F$ be the finite graph which is $C_+[\{0,1\} \times [0,7]]$ along with one more vertex, $\infty$, which is adjacent to $(0,7)$ and $(1,7)$, only. Provided for any $4$ points on distinct edges of $C_+[\{0,1\} \times [0,7]]$ there is an arc containing them which ends at $\infty$ or $(0,1)$, we are done.

Since this is a finite graph this can be verified by hand. There are ${21 \choose 4} =5985$ choices of $4$ points from the $21$ edges of $C_+[\{0,1\} \times [0,7]]$. The following python~2.7 program confirms the desired statement:

\begin{lstlisting}[language=Python]
from itertools import combinations, ifilter


# find all arcs extending a given one, a, that ends at vertex v
def all_arcs_fromv_exta(v,a):
    arcs=[a]
    for w in E[v]:
        if not(w in a): 
            arcs=arcs+all_arcs_fromv_exta(w,a+[w])

    return arcs

# all arcs starting at the vertex v
def all_arcsfromv(v):
    return all_arcs_fromv_exta(v,[v])

# if p is a point (in fact an edge) then split that edge in two, 
# and call p the new vertex to get a new graph.
def split_edge(p,V,E):

    n=len(V)
    V=V+[n] 
    # the new vertex is given the next available number, n
    # recall p is an edge with endpoints v=p[0] and w=p[1]
    v,w=p[0],p[1]
    E[v]=[x for x in E[v] if (x<>w)]+[n] 
    # v keeps its old neighbors, except w is removed and n added
    E[w]=[x for x in E[w] if (x<>v)]+[n] 
    # and similarly for w
    E=E+[[v,w]] 
    # and the new vertex, n (at the end of E) has v and w as neighbors
    return (V,E)

# all edges as (vertex,vertex) pairs - ordered to remove repeats
def all_edges(E):
    return [ (a,b) for a in V for b in E[a] if (a<b)] 

# dont want to consider points taken from the edges going to infinity (15)
def keepgood(P):
    return not( ((13,15) in P) or ((14,15) in P) ) 
      



# global

# 0 is b (the loop vertex) and 15 is infinity
V=[0,1,2,3,4,5,6,7,8,9,10,11,12,13,14,15] 

#E[v] lists the vertices adjacent to vertex numbered v.
# so E[1] is [3,2] says vertex 1 (top left) is adjacent to 3 
# (to its right) and 2 (below it).
E=[ [6,8,2], [3,2], [1,4,0], [1,4,5], [2,3,6], [3,6,7], [4,5,0], [5,8,9], [0,7,10], [7,10,11], [8,9,12], [9,12,13], [10,11,14], [11,14,15], [12,13, 15], [13,14] ]

allE=all_edges(E)

# choose 4 edges (points) from all the edges 
# excluding the two going to infinity
for P in ifilter(keepgood,combinations(allE,4)):

#We are about to modify the base graph to add new vertices 
# corresponding to the points in P.
#This modification will still be in place as we iterate. 
# So we reset to the base graph here.
    V=[0,1,2,3,4,5,6,7,8,9,10,11,12,13,14,15] 
    E=[ [6,8,2], [3,2], [1,4,0], [1,4,5], [2,3,6], [3,6,7], [4,5,0], [5,8,9], [0,7,10], [7,10,11], [8,9,12],[9,12,13], [10,11,14], [11,14,15], [12,13, 15], [13,14] ]

#As promised, at each point in P add a vertex to the base graph
#Note that they get vertex numbers: 16, 17, 18 and 19    
    for p in P:
        V,E=split_edge(p,V,E)

#We have not yet found an arc from infinity  1 containing all 
# the points in P so...
    good=False
# Look at each arc from infinity (15) or 1 (which is on e)
    for a in all_arcsfromv(15)+all_arcsfromv(1):

        if (set((16,17,18,19)) <= set(a)): 
        # if every point in P (16,17,18,19) is in this arc 
        # then good! and can stop checking this P
                good=True
                break
# Announce the news for P.
    if (good==True): 
        print 'Points ',P, ' lie on arc ',a

    if (good==False):
        print 'Points ', P, ' DO NOT lie on any arc!'

\end{lstlisting}

\fi

\end{document}